\documentclass[11pt]{article}

\title{Endnotes using pagenote package at the end of each chapter}

\usepackage{pagenote, hyperref, url}

\renewcommand*{\pagenotesubhead}[1]{}

\makepagenote

\usepackage{epsfig,amsmath,amsthm, wasysym, xcolor}

\usepackage{amsmath, amssymb}

\usepackage{graphicx, tikz}
\usetikzlibrary{calc,shapes}
\tikzset{mynode/.style={circle, thick,draw},}
\tikzstyle{every picture}=[->,>=latex]
\usepackage{float}

\newtheorem{df}{Definition}[section]

\newtheorem{theorem}[df]{Theorem}
\newtheorem{lemma}[df]{Lemma}

\newtheorem{proposition}[df]{Proposition}

\usepackage[left=1in, right=1in, top=1.5in, bottom=1in]{geometry}
\usetikzlibrary{decorations.pathreplacing}
\usepackage{array}
\usepackage{algorithm}
\usepackage[noend]{algorithmic}

\title{Algorithmic Applications of Tyshkevich's Graph Decomposition: \\ A Primer and a Toolkit}
\author{Christine T. Cheng and Chelsea Ann Lambert}
\date{Department of Electrical Engineering and Computer Science \\ University of Wisconsin-Milwaukee\\ {\tt ccheng@uwm.edu; lambe222@uwm.edu}%
       }

\begin{document}
\maketitle

\begin{abstract}
A graph that is completely determined by its degree sequence is called a {\it unigraph}.  In 2000, Regina Tyshkevich published one of the most important papers on unigraphs \cite{Ty00}.  There are two parts to the paper: a decomposition theorem that describes how every graph can be broken into a sequence of basic graphs and  a complete classification of all  basic unigraphs. Together, they reveal how every unigraph is constructed. 

We provide an informal overview of Tyshkevich's  results and show how they enable the computation of  various graph parameters of unigraphs in linear time.  We also provide a toolkit (\url{https://chelseal11.github.io/tyshkevich_decomposition_toolkit/}) that implements  the algorithms described in this write-up.



\end{abstract}

\section{Introduction}

	Every graph $G$ has a {\it degree sequence} $D_G$ that lists  the degrees of the vertices from largest to smallest.  For example,  the degree sequence of the tree $T$ in Figure \ref{fig1} is $(3,2,1,1,1) = (3,2,1^3)$.   
In general,  a graph's degree sequence is not unique to it;  two or more graphs can share the same degree sequence.  
Take the $8$-cycle, $C_8$.  It has the same degree sequence as $C_5 \cup C_3$ and $2C_4$,  which is $(2,2,2,2,2, 2, 2, 2) = (2^8)$.    

A graph that is completely determined by its degree sequence is called a {\it unigraph}.   When $G$ is a unigraph and $D_{G} = D_{G'}$, then $G$ and $G'$ are isomorphic.  Some examples of unigraphs are complete graphs, $mK_2$ ($m$ copies of $K_2$), $C_5$ and their complements.  It may seem that unigraphs are special so there are not many of them.  But, in fact, there are between $2.3^{n-2}$ and $2.6^n$ (unlabeled) unigraphs with $n$ vertices \cite{Ty00}!   
	

One of the most important papers on unigraphs is {\it ``Decomposition of Graphical Sequences and Unigraphs"} by Regina Tyshkevich, published in the journal {\it Discrete Mathematics} in 2000 \cite{Ty00}.  Tyshkevich had been working on unigraphs with several co-authors since the late 1970's (see references 15 to 23 in \cite{Ty00}).  The papers were in Russian  and remained relatively unknown in the West.  The publication of the decomposition paper finally brought Tyshkevich's work to a wider audience.

\begin{figure}[t]
\begin{center}
\begin{tikzpicture}
{
\begin{scope}[every node/.style={circle,thick,draw}]

    \node(a) at (0,0.5) {$e$};
    \node(b) at (0, 1.5) {$c$};
    \node(c) at (0, 2.5) {$a$};
    \node(d) at (1.5, 1) {$d$};
    \node(e) at (1.5, 2.5) {$b$};
\end{scope}

\path[-, draw, thick] (a) -- (d);
\path[-, draw, thick] (b) -- (d);
\path[-, draw, thick] (c) -- (e);
\path[-, draw, thick] (d) -- (e);
}
\node(A) at (1.5, -0.5) {$A$};
\node(B) at (0, -0.5) {$B$};

\end{tikzpicture}
\hspace*{2em}
\begin{tikzpicture}
{
\begin{scope}[every node/.style={circle,thick,draw}]

    \node(e) at (0,0.5) {$e$};
    \node(c) at (0, 1.5) {$c$};
    \node(a) at (0, 2.5) {$a$};
    \node(x) at (3.5, 2.5) {$x$};
    \node(y) at (3.5,1.5) {$y$};
    \node(z) at (3.5, 0.5) {$z$};
    \node(d) at (1.5, 1) {$d$};
    \node(b) at (1.5, 2.5) {$b$};
\end{scope}
\path[-, draw, thick] (a) -- (b);
\path[-, draw, thick] (c) -- (d);
\path[-, draw, thick] (e) -- (d);
\path[-, draw, thick] (b) -- (d);
\path[-, draw, thick] (y) -- (z);
\path[-, draw, thick] (x) -- (y);
\path[-, draw, thick] (b) -- (x);
\path[-, draw, thick] (b) -- (y);
\path[-, draw, thick] (b) -- (z);
\path[-, draw, thick] (d) -- (x);
\path[-, draw, thick] (d) -- (y);
\path[-, draw, thick] (d) -- (z);
\draw[-][-, thick] (x) to[out=330, in=30] (z);
\node(A) at (1.5, -0.5) {$(T, A,B)\circ H$};
}
\end{tikzpicture}
\end{center}
\caption{The tree $T$ on the left has degree sequence $(3,2,1,1,1) = (3, 2, 1^3)$.  It has one $KS$-partition $(\{b,d\},  \{a,c,e\})$.   
It is combined with the another graph $H$, a $3$-cycle, using the composition operation to create the graph $(T, A, B) \circ H$.} 
\label{fig1}
\end{figure}
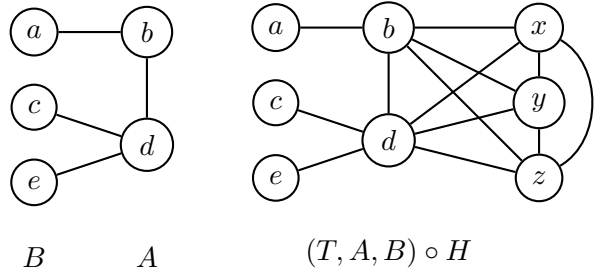
	

There are two parts to Tyshkevich's paper:  (i) a decomposition theorem that describes how every graph can be broken  into a sequence of basic components and   (ii) a complete classification of all the basic unigraph components. Together, they reveal how every unigraph is generated.   
The goal of this write-up is to provide an informal overview of Tyshkevich's  results and show how they can be used to compute various graph parameters.  

The algorithms are a variation of the following idea:  to compute the parameter $f(G)$,  determine $f(G_i)$ for each basic component $G_i$ and then combine them to obtain $f(G)$.   In Section \ref{sec4},  we present a result by Whitman \cite{Whitman24} and extend it to show that when $G$ is a unigraph,  the clique number $\omega(G)$, the independence number $\alpha(G)$, the vertex cover number $\beta(G)$ and the chromatic number $\chi(G)$ can all be computed in linear time.  In Section \ref{sec5},  we describe how a modified version of Tyshkevich's decomposition theorem can be used to characterize the automorphism group of a graph. Consequently, when $G$ is a unigraph, the distinguishing number and fixing number of $G$, two parameters associated with symmetry-breaking of graphs,  can also be computed in linear time.  We accompany our exposition with a toolkit to make the material more accessible to the reader.  The toolkit is described  in Section \ref{sec6}.



We note that other researchers have also addressed algorithmic problems on unigraphs.  Calamoneri and Petreschi \cite{CaPe11} used the above template to design an approximation algorithm that computes an $L(2,1)$-labeling of a unigraph.  The number of colors used  is at most $1.5 \times OPT$.   Nakahata \cite{Na22} utilized the same idea to prove that the cliquewidth of a unigraph is at most $4$.   What makes the algorithms  in Sections \ref{sec4} and \ref{sec5} noteworthy is that they output the exact answers in linear time.


 
\section{The Decomposition Theorem}
\label{sec2}

	{\it Split graphs} play a prominent role in Tyshkevich's decomposition theorem.  A graph $G$ is {\it split} if its vertex set can be partitioned into two sets so that the first set of vertices induce a clique while the second set of vertices induce a stable or independent set.  The two sets form what's called a {\it $KS$-partition} of $G$. The $K$-part is for the clique and the $S$-part is for the stable set.   
	
	Consider the tree in Figure \ref{fig1}.  It has a $KS$-partition $(\{b,d\}, \{a, c, e\})$, so it is a split graph.  In fact, it is the tree's only $KS$-partition.  On the other hand, the complete graph $K_n$  has two types of $KS$-partitions:  the first one  has the entire vertex set in the $K$-part and an empty set for the $S$-part; the second one has all but one vertex in the $K$-part and a single vertex in the $S$-part.  Hammer and Simeone \cite{HaSi81} describes the $KS$-partitions of a split graph. 
	
\begin{theorem}
{\bf \cite{HaSi81}} For any $KS$-partition $(A,B)$ of split graph $G$, exactly one of the following holds:
\begin{itemize}
\item $|A| = \omega(G)$ and $|B| = \alpha(G)$
\item $|A| = \omega(G)$ and $|B| = \alpha(G)-1$ (called $K$-max)
\item $|A| = \omega(G) -1$ and $|B| = \alpha(G)$ (called $S$-max)

\end{itemize}
where $\omega(G)$ and $\alpha(G)$ are the clique and independence numbers of $G$ respectively.  Moreover, in a $K$-max partition, there  is a vertex $u \in A$ so that $B \cup \{u\}$ induces a stable set while in an $S$-max partition, there is a vertex $v \in B$ so that $A \cup \{v\}$ induces a complete graph.  \end{theorem}

Split graphs that have the first kind of $KS$-partition are  called {\it balanced}; otherwise, they are {\it unbalanced}.  It turns out that a balanced split graph has only one $KS$-partition \cite{ChCoTr16},  but an unbalanced split graph has both a $K$-max and a $S$-max partition because of the existence of the vertices $u$ and $v$, which are called {\it swing vertices}.  Thus, the tree in Figure \ref{fig1} is a balanced split graph while $K_n$ is an unbalanced one.


\begin{proposition}
\label{uniqueKSprop}
{\bf \cite{ChCoTr16}} A balanced split graph $G$ has a unique $KS$-partition.  
\end{proposition}

	
	
	For a split graph $G$, we often want to specify the $KS$-partition of interest,  so we express the graph as $(G,A, B)$ where $A$ is the $K$-part and $B$ is the $S$-part of the $KS$-partition.  The degree sequence of $(G,A,B)$ is also expressed as a {\it paired degree sequence} $(d_1, \hdots, d_r; d_{r+1}, \hdots, d_n)$ to distinguish between the degrees of the vertices in $A$ and $B$. 
	

	
\medskip	
\bigskip 
\centerline{$*  *  *  * *$}
\bigskip 
\medskip


	Let $(G,A,B)$ be a split graph.  Let $H$ be another graph.  Tyshkevich defined the {\it composition of $(G,A,B)$ and $H$} as the graph $(G,A,B) \circ H$. Its vertex set is $V(G) \cup V(H)$ and its edge set is $E(G) \cup E(H) \cup \{uv|u \in A, v \in V(H)\}$.   That is,  $(G,A,B) \circ H$  is the graph $G \cup H$ together with the edges between the vertices in $A$ and the vertices in $H$.  Figure \ref{fig1} has an example.  
	




\begin{figure}[t]
\bigskip
\begin{center}
\begin{tikzpicture}
{
\begin{scope}[every node/.style={circle,thick,draw}]

    \node(v3) at (0,2.5) {$v_3$};
    \node(v4) at (0, 0.5) {$v_4$};
    \node(v1) at (1.5, 2.5) {$v_1$};
    \node(v2) at (1.5, 0.5) {$v_2$};
    \node(v0) at (2.5,1.5) {$v_0$};
 \end{scope}
\path[-, draw, thick] (v3) -- (v4);
\path[-, draw, thick] (v3) -- (v1);
\path[-, draw, thick] (v3) -- (v2);
\path[-, draw, thick] (v3) -- (v0);
\path[-, draw, thick] (v4) -- (v1);
\path[-, draw, thick] (v4) -- (v2);
\path[-, draw, thick] (v4) -- (v0);
}
\end{tikzpicture}

\end{center}
\caption{The threshold graph $(\{v_4\}, \{v_4\}, \emptyset) \circ (\{v_3\}, \{v_3\}, \emptyset) \circ  (\{v_2\}, \emptyset, \{v_2\}) \circ  (\{v_1\}, \emptyset, \{v_1\})  \circ \{v_0 \}$.}
\label{fig3}
\end{figure}
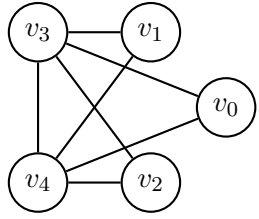
	
	The composition operation can be applied to two split graphs $(G', A', B') \circ (G, A, B)$ resulting in another split graph $(G'', A \cup A', B \cup B')$.  It can also be applied 
multiple  times to produce new graphs; e.g.,  $(G_k , A_k, B_k) \circ (G_{k-1}, A_{k-1}, B_{k-1})  \circ \hdots \circ (G_1, A_1, B_1) \circ G_0$.    Interestingly, the composition operation can be applied in {\it any} order because the operation is associative.


One important class of graphs that can be described this way are {\it threshold graphs}.  They are the graphs formed when each $G_i$, including $G_0$, is a single-vertex graph.   For simplicity, we denote a single-vertex graph as $\{v\}$.  When it is treated as a split graph and the vertex is in the $K$-part, we refer to it as $(\{v\}, \{v\}, \emptyset)$; its paired degree sequence is $(0; \emptyset)$.  On the other hand, 
 when the vertex is in the $S$-part, we refer to it as $(\{v\}, \emptyset, \{v\})$ and its degree sequence is $(\emptyset, 0)$.
 For example, 
$$(\{v_4\}, \{v_4\}, \emptyset) \circ (\{v_3\}, \{v_3\}, \emptyset) \circ  (\{v_2\}, \emptyset, \{v_2\}) \circ  (\{v_1\}, \emptyset, \{v_1\})  \circ \{v_0 \} $$
results in the graph shown in Figure \ref{fig3}.    

\medskip	
\bigskip 
\centerline{$*  *  *  * *$}
\bigskip 
\medskip


Tyshkevich defined a graph as {\it decomposable} if it can be obtained as the composition of some split graph and another graph; otherwise, it is an {\it indecomposable graph}.  

How do we know if a graph is decomposable?  The answer, it turns out, lies in a graph's degree sequence!  Let us examine the composition operation again. In $(G, A, B) \circ H$, all the vertices in $A$ become adjacent to all the vertices in $H$. 
Thus,
\begin{itemize}
\item when $u \in A$, $deg(u) \ge |V(H)| + |A| -1 $ since $u$ is also adjacent to all the other vertices in $A$; 
\item when  $v \in V(H)$, $|A| \leq deg(v) \leq |A| + |V(H)|-1$ since the number of neighbors of $v$ in $H$ range from $0$ to $|V(H)|-1$;
\item when $w \in B$, $0 \leq deg(w) \leq |A|$ since the number of neighbors of $w$ in $G$ range from $0$ to $|A|$.
\end{itemize}
 It follows that in the degree sequence of  $(G, A, B) \circ H$, the degrees of the vertices in $A$ appear first followed by those in $H$ and finally by those in $B$.  
  Another way of putting it is  if the degree sequence of $(G, A, B) \circ H$ is $(d_1, d_2, \hdots, d_N)$, then   there are two indices $p$ and $q$ with $p + q < N$ so that $$(d_1, \hdots, d_p), (d_{p+1}, \hdots, d_{N-q}) \mbox{ and } (d_{N-q+1}, \hdots, d_N)$$
 are the degrees of the vertices in $A, H$ and $B$ respectively.   

Following the  notation above, we also note that  in $(G, A, B) \circ H$, a vertex $u \in A$  is adjacent  to (i) $p-1$ other vertices in $A$, (ii) all the vertices in $H$ and (iii) its neighbors in $B$.  Hence, if $v_i$ is the vertex with degree $d_i$, and $deg_B(v_i)$ is the number of neighbors of $v_i$ in $B$, 
\begin{eqnarray*}
  \sum_{i=1}^p d_i & = &  \sum_{i=1}^p \left[(p-1) + (N-p-q) + deg_B(v_i) \right] \\
                               & = & p(p-1) + p(N-p-q) + \sum_{i=1}^p deg_B(v_i)  \\
                               & = &  p(N-q-1) +  \sum_{i = N-q+1}^N d_i. 
\end{eqnarray*}
Tyshkevich showed that these observations are both necessary and sufficient for proving that a graph is decomposable.


 \begin{theorem} {\bf \cite{Ty00}}
   A graph $G^*$  with degree sequence $(d_1, d_2,$ $\hdots$, $d_N)$  is decomposable if and only if  there exists indices $p, q$ with $0< p + q < N$ so that $$ \sum_{i=1}^p d_i = p(N-q-1) + \sum_{i = N-q+1}^N d_i.$$
 If $G^*$ is decomposable, then  $G^* = (G, A, B) \circ H$ where   $(d_1, \hdots, d_p)$, $(d_{p+1}, \hdots, d_{N-q})$ and $(d_{N-q+1}, \hdots, d_N)$ are the degrees of the vertices in $A$, $H$ and $B$ respectively.   
 \label{mainthm1}
 \end{theorem}

\medskip	
\bigskip 
\centerline{$*  *  *  * *$}
\bigskip 
\medskip


When $G^*$ is the composition of many graphs,  there are many ways of expressing $G^*$ as the composition of two graphs because of the associativity of the composition operation.  That is, there are many pairs of indices $(p,q)$  that can satisfy the conditions in Theorem \ref{mainthm1}.  But if we choose  the pair with the smallest $p$ value, then $(G,A,B)$ is {\it guaranteed} to be an { indecomposable} graph.  We can then decompose $H$ further.  In this way, we obtain a decomposition of $G^*$ into indecomposable graphs.   This leads us to Tyshkevich's main result.

\begin{theorem} {\bf (The Canonical Decomposition Theorem of Tyshkevich \cite{Ty00})}
Every graph $G^*$ can be represented as a composition 
$$ G^*= (G_k , A_k, B_k) \circ (G_{k-1}, A_{k-1}, B_{k-1})  \circ \hdots \circ (G_1, A_1, B_1) \circ G_0$$
where each $G_i$ is an indecomposable graph.  This decomposition is canonical; that is, let
$$G' = (G'_\ell, A'_\ell, B'_\ell) \circ (G'_{\ell-1}, A'_{\ell-1}, B'_{k-1}) \circ \hdots \circ (G'_1, A'_1, B'_1) \circ G'_0.$$
Then $G^* \cong G'$ if and only if $k = \ell$, $G_0 \cong G'_0$ and $(G_i, A_i, B_i) \cong (G'_i, A'_i, B'_i)$ for $i = 1, \hdots, k$. 
\label{mainthm2}
\end{theorem}

Using Theorem \ref{mainthm2}, we presented a linear-time algorithm called {\tt Decompose($G^*$)} in \cite{Ch25} that outputs a stack that contains the (paired) degree sequence of $G_0$ on top followed by the paired degree sequences of  $G_1, \hdots, G_k$.

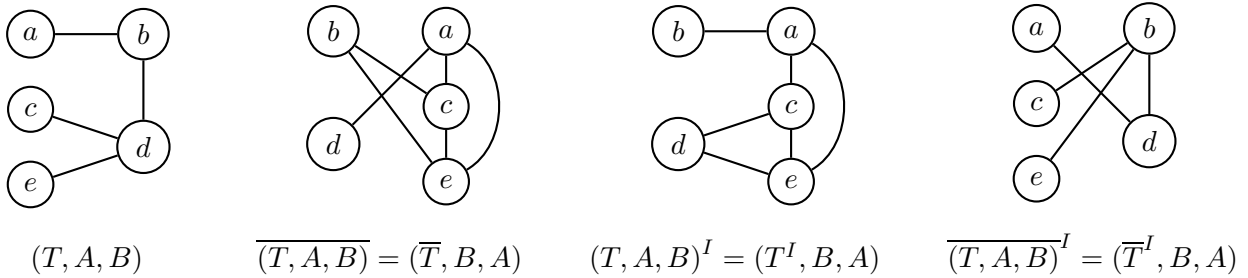
\begin{figure}[t]
\bigskip

\begin{center}
\begin{tikzpicture}
{
\begin{scope}[every node/.style={circle,thick,draw}]

    \node(a) at (0,0.5) {$e$};
    \node(b) at (0, 1.5) {$c$};
    \node(c) at (0, 2.5) {$a$};
    \node(d) at (1.5, 1) {$d$};
    \node(e) at (1.5, 2.5) {$b$};
  \end{scope}
  \node(T) at (0.75, -0.5) {$(T, A, B)$};

\path[-, draw, thick] (a) -- (d);
\path[-, draw, thick] (b) -- (d);
\path[-, draw, thick] (c) -- (e);
\path[-, draw, thick] (d) -- (e);
}
\end{tikzpicture}
\hspace*{2em}
\begin{tikzpicture}
{
\begin{scope}[every node/.style={circle,thick,draw}]

    \node(e) at (0,0.5) {$e$};
    \node(c) at (0, 1.5) {$c$};
    \node(a) at (0, 2.5) {$a$};
    \node(d) at (-1.5, 1) {$d$};
    \node(b) at (-1.5, 2.5) {$b$};
\end{scope}
\node(T) at (-0.75, -0.5) {$\overline{(T, A, B)} = (\overline{T}, B, A)$};

\path[-, draw, thick] (a) -- (c);
\draw[-][-, thick] (a) to[out=330, in=30] (e);
\path[-, draw, thick] (c) -- (e);
\path[-, draw, thick] (b) -- (c);
\path[-, draw, thick] (b) -- (e);
\path[-, draw, thick] (a) -- (d);
}
\end{tikzpicture}
\hspace*{1em}
\begin{tikzpicture}
{
\begin{scope}[every node/.style={circle,thick,draw}]

    \node(e) at (1.5,0.5) {$e$};
    \node(c) at (1.5, 1.5) {$c$};
    \node(a) at (1.5, 2.5) {$a$};
    \node(d) at (0, 1) {$d$};
    \node(b) at (0, 2.5) {$b$};
\end{scope}
\node(T) at (0.75, -0.5) {${(T, A, B)}^I = ({T^I}, B, A)$};

\path[-, draw, thick] (a) -- (b);
\path[-, draw, thick] (c) -- (d);
\path[-, draw, thick] (e) -- (d);
\path[-, draw, thick] (a) -- (c);
\draw[-][-, thick] (a) to[out=330, in=30] (e);
\path[-, draw, thick] (c) -- (e);

}
\end{tikzpicture}
\hspace*{1em}
\begin{tikzpicture}
{
\begin{scope}[every node/.style={circle,thick,draw}]

    \node(e) at (0,0.5) {$e$};
    \node(c) at (0, 1.5) {$c$};
    \node(a) at (0, 2.5) {$a$};
    \node(d) at (1.5, 1) {$d$};
    \node(b) at (1.5, 2.5) {$b$};
\end{scope}
\node(T) at (0.75, -0.5) {$\overline{(T, A, B)}^I = (\overline{T}^I, B, A)$};

\path[-, draw, thick] (b) -- (d);
\path[-, draw, thick] (a) -- (d);
\path[-, draw, thick] (b) -- (c);
\path[-, draw, thick] (b) -- (e);

}
\end{tikzpicture}

\end{center}
\caption{Given the tree $T$ on the left with $A = \{b,d\}$ and $B = \{a, c, e\}$, the complement,  the inverse and the inverse of the complement of $(T, A, B)$ are shown on the right.}
\label{fig4}
\end{figure}

\section{Indecomposable Unigraphs}
\label{sec3}

In the second part of her paper, Tyshkevich noted the implication of Theorem \ref{mainthm2} for unigraphs. 


\begin{theorem}{\bf \cite{Ty00}}
 $G$ is a unigraph if and only if each indecomposable component in the canonical decomposition of $G$ is also a unigraph.
 \label{thmcompose}
\end{theorem}

Equally  important, Tyshkevich described {\it all} the indecomposable unigraphs.  The proof is substantial; it is twenty pages long with many cases to sort through.  Tyshkevich was able to keep the list of indecomposable unigraphs manageable because she made use of the complements and inverses of graphs. 

  
  Recall that the  {\it complement} of $G$, $\overline{G}$, has the same vertex set as $G$ and two vertices are adjacent if and only if they are not adjacent in $G$. When $G$ is a split graph with $KS$-partition $(A,B)$,  the complement of $(G,A,B)$ is $\overline{(G,A,B)} = (\overline{G}, B, A)$. Its {\it inverse} is $(G,A,B)^I = (G^I, B, A)$ where $G^I$ is obtained from $G$ by deleting the edges $\{aa': a, a' \in A\}$ and adding the edges $\{bb': b, b' \in B\}$.  The two operations can then be combined to  create the {\it complement of the inverse} of $(G, A, B)$, which is also the {\it inverse of the complement} of $(G, A, B)$.  Figure \ref{fig4} shows the two graph operations.








The  indecomposable non-split unigraphs are described using four types of  graphs.  The first two  are $C_5$ and $mK_2$.   The third one is  $U_2(m, \ell)$, the graph formed by the disjoint union of $mK_2$ and the star $K_{1, \ell}$.  The fourth one is $U_3(m)$, the graph formed by a center vertex, $mK_2$ and $P_3$.  The center vertex is adjacent to all the vertices of $mK_2$, creating $m$ triangles, and to the endpoints of $P_3$, creating a 4-cycle.   These non-split unigraphs are illustrated in Table \ref{UnsplittableTable}.

\begin{theorem}
 \label{Tyshkevich-unsplittable}
 {\bf \cite{Ty00}} Let $G$ be an indecomposable non-split unigraph with two or more vertices.  Then $G$ or $\overline{G}$ is one of the following graphs: 
 $C_5, m K_2$ with $m \ge 2$, $ U_2(m,\ell)$ with $m \ge 1, \ell \ge 2$ or $U_3(m)$ with $m \ge 1$. 
 \end{theorem}

 The indecomposable split unigraphs, on the other hand, are described using five types of graphs.  The first one is the single-vertex graph.   The next one is $S(p,q)$, the graph that consists of $q$ copies of the star $K_{1,p}$ whose $q$ centers form a complete graph.  The third one is $S_2(p_1, q_1, p_2, q_2, \cdots, p_m, q_m)$ which  is a generalization of $S(p,q)$; it consists of $q_i$ copies of $K_{1, p_i}$ for $i = 1, \hdots, m$ and the centers of all the stars form a complete graph.   The fourth graph is $S_3(p, q_1, q_2)$ which is made up of $S(p, q_1)$ and $S(p+1, q_2)$ and all their star's centers are pairwise adjacent to each other.  Additionally, there is a vertex $e$ that is adjacent to the centers of the stars in $S(p, q_i)$ only.    The last one is $S_4(p,q)$ which is made up of $S_3(p, 2, q)$ and a vertex $f$ that is adjacent to all the vertices of $S_3(p,2,q)$ except for $e$.  These split unigraphs are illustrated in Table \ref{SplittableTable}.

 \begin{theorem}
 {\bf \cite{Ty00}} Let $G$ be an indecomposable split unigraph.  Then $G$,  $\overline{G}$, $G^I$ or $\overline{G^I}$ is one of the following: 
a single-vertex graph,  $S(p,q)$ with $p \ge 1, q \ge 2, $  $ S_2(p_1, q_1, p_2, q_2, \cdots, p_m, q_m)$ with $m \ge 2$ and $p_1 > p_2 > \hdots > p_m,$   $ S_3(p, q_1, q_2)$ with $p \ge 1, q_1 \ge 2, q_2 \ge 1, \;$  or $\;S_4(p,q)$ with $p \ge 1, q \ge 1$.
  \label{Tyshkevich-splittable}

 \end{theorem}

When $\overline{G} = H$, then $G = \overline{H}$.  Similarly, when $G^I = H$, then $G = H^I$.  Thus, another way of stating Theorems \ref{Tyshkevich-unsplittable} and \ref{Tyshkevich-splittable} is that when $G$ is an indecomposable non-split unigraph,  $G$ is isomorphic to one of the four types graphs in Table \ref{UnsplittableTable} or their complements.  On the other hand,  when $G$ is an indecomposable split graph, $G$ is isomorphic to one of the five types graphs in Table \ref{Splittable}, their complements, inverses or complements of the inverses.  Thus, whenever $G$ is an indecomposable unigraph, we can refer to its {\it type} (and related parameters). For example,  $C_4$  is  isomorphic to $\overline{2K_2}$ so its type is {\it the complement of $mK_2$ with $m = 2$}.  On the other hand, the type of the tree in Figure \ref{fig1} is {\it $S_2(p_1, q_1, p_2, q_2)$ with $p_1 = 2, q_1 = p_2 = q_2 = 1$}.  When $G$ is a single-vertex split graph, we refer to its type as $K_1$ or $S_1$ if the vertex is in the $K$-part or in the $S$-part of the split graph respectively.


A subtle but important consequence of Theorems \ref{thmcompose} to \ref{Tyshkevich-splittable} is that unigraphs and their indecomposable components can be recognized in linear time.  The fact that unigraphs can be recognized efficiently is not new; Kleitman and Li \cite{KL75}  and Borri et al.~\cite{Borri2011} have already designed linear-time algorithms.   What {\it is} new is that their indecomposable components  can also be identified quickly.  

\begin{theorem}
Let $G$ be a graph.  There is a linear-time algorithm that determines if $G$ is a unigraph.  Moreover, when $G$ is a unigraph, it outputs the type (and related parameters) of each indecomposable component of $G$.
\end{theorem}

  The theorem holds because the graphs in Tables \ref{UnsplittableTable} and  \ref{SplittableTable} have very specific patterns for their degree sequences.  We provide a full description of the linear-time algorithm called {\tt IsUnigraph($G$)} in the appendix.   The algorithm is based on our work in \cite{Ch25}, which is similar to that of Calamoneri and Peterschi \cite{CaPe11}. 

 Tyskevich's results can also be used to generate unigraphs. 
Prior to her work, it was unclear how one would go about creating, say, a hundred unigraphs that are not from the same subfamily like the threshold graphs.  Now we know because she has identified the  basic pieces  for assembling unigraphs.   Below are unigraphs with ten vertices.  Some are indecomposable themselves or formed by composing an  indecomposable unigraph with any number of indecomposable split unigraphs.  

\begin{itemize}
\item $8^{10}$  from $\overline{5 K_2}$
\item $(3^1 , 1^9)$ from $U_2(3,3)$
\item  $(8^4 , 5^4 , 2^2)$  from $S(2,2)^I \circ 2K_2$
\item   $(8^4 , 7^1 , 6^2 , 3^3)$ from $\overline{S_3(1,2,1)}  \circ S_1 \circ S_1$   
\item $(7^4 , 6^1 , 5^1 , 3^3 , 0^1)$ from $S_1 \circ \overline{S_3(1,2,1)}  \circ S_1$
\item $(9^1 , 7^1 , 6^1 , 4^5 , 1^2)$ from $K_1 \circ S_1 \circ S_1 \circ K_1 \circ U_3(1)$ 
\end{itemize}

Drawings and adjacency lists of the above graphs can be obtained using the  Havel-Hakimi algorithm \cite{Havel1955, Hakimi1962}.    
  \newpage
  
  \vspace*{0.5in}
  
   \begin{table}[h]
 \begin{center}
\small
\label{UnsplittableTable}
 \begin{tabular}{|>{\centering\arraybackslash}m{1.5in}|>{\centering\arraybackslash}m{2.5in}|>{\centering\arraybackslash}m{0.5in}|>{\centering\arraybackslash}m{0.5in}| }
 \hline
{Indecomposable Non-split Unigraphs $G$} & Illustration and Degree Sequence  & $\omega(G)$ & $\alpha(G)$  \\
\hline
$C_5$ & 
\vspace{1em}
\begin{tikzpicture}
    \node[circle,draw,fill=lightgray] (v1) at (90:1) {};    
    \node[circle,draw,fill=lightgray] (v2) at (162:1) {};   
    \node[circle,draw,fill=lightgray] (v3) at (234:1) {};   
    \node[circle,draw,fill=lightgray] (v4) at (306:1) {};   
    \node[circle,draw,fill=lightgray] (v5) at (18:1) {};  
    \draw[-] (v1) -- (v2);
    \draw[-] (v2) -- (v3);
    \draw[-] (v3) -- (v4);
    \draw[-] (v4) -- (v5);
    \draw[-] (v5) -- (v1);
\end{tikzpicture}

$(2^5)$  
&  $2$ & $2$ \\
\hline
$mK_2$ with $m \ge 2$ &

\vspace{1em}
\begin{tikzpicture}
    \node[circle,draw,fill=lightgray] (a1) at (0,1) {};
    \node[circle,draw,fill=lightgray] (b1) at (0,0) {};
    \draw[-] (a1) -- (b1);
    
    \node[circle,draw,fill=lightgray] (a2) at (0.75,1) {};
    \node[circle,draw,fill=lightgray] (b2) at (0.75,0) {};
    \draw[-] (a2) -- (b2);
    
    \node at (1.5,0.5) {$\cdots$};
    \node[circle,draw,fill=lightgray] (a3) at (2.25,1) {};
    \node[circle,draw,fill=lightgray] (b3) at (2.25,0) {};
    \draw[-] (a3) -- (b3);
    
    \draw[-][decorate,decoration={brace,mirror,amplitude=10pt}]
        ([yshift=-5pt]b1.south west) -- ([yshift=-5pt]b3.south east)
        node[midway,yshift=-15pt]{$m$};
\end{tikzpicture}  

 $(1^{2m})$   & 2 & $m$\\
\hline
$U_2(m,\ell) = mK_2 \cup K_{1,\ell}$   with $m \ge 1, \ell \ge 2$
& 
\vspace{1em}
\begin{tikzpicture}
    \node[circle,draw,fill=lightgray] (a1) at (0,1) {};
    \node[circle,draw,fill=lightgray] (b1) at (0,0) {};
    \draw[-] (a1) -- (b1);
    \node[circle,draw,fill=lightgray] (a2) at (0.75,1) {};
    \node[circle,draw,fill=lightgray] (b2) at (0.75,0) {};
    \draw[-] (a2) -- (b2);
    \node at (1.4,0.5) {$\cdots$};
    \node[circle,draw,fill=lightgray] (a3) at (2.0,1) {};
    \node[circle,draw,fill=lightgray] (b3) at (2.0,0) {};
    \draw[-] (a3) -- (b3);   
    \draw[-][decorate,decoration={brace,mirror,amplitude=10pt}]
        ([yshift=-5pt]b1.south west) -- ([yshift=-5pt]b3.south east)
        node[midway,yshift=-15pt]{$m$};
      \node[circle,draw,fill=lightgray] (a4) at (4,1) {};
      \node[circle,draw,fill=lightgray] (d1) at (3,0) {};
      \draw[-] (a4) -- (d1);
       \node[circle,draw,fill=lightgray] (d2) at (3.5,0) {};
       \draw[-] (a4) -- (d2);
       \node[circle,draw,fill=lightgray] (d3) at (4.5,0) {};
       \draw[-] (a4) -- (d3);
       \node[circle,draw,fill=lightgray] (d4) at (5.0,0) {};
       \draw[-] (a4) -- (d4);
       \node at (4.05,0) {$\cdots$};
      \draw[-][decorate,decoration={brace,mirror,amplitude=10pt}]
        ([yshift=-5pt]d1.south west) -- ([yshift=-5pt]d4.south east)
        node[midway,yshift=-15pt]{$\ell$};
\end{tikzpicture}

   $(\ell, 1^{2m+\ell})$  & 2 & $m + \ell$ \\
\hline
$U_3(m)$, $m \ge 1$ &  
\vspace{1em}
\begin{tikzpicture}[scale=1]

    \node[circle,draw,fill= lightgray] (a) at (-2,1) {};
    \node[circle,draw,fill=lightgray] (b) at (-3,0) {};
    \node[circle,draw,fill=lightgray] (c) at (-2,-1) {};
    \node[circle,draw,fill=lightgray] (d) at (-1,0) {};
    
    \draw[-] (a) -- (b) -- (c) -- (d) -- (a);


    \node[circle,draw,fill=lightgray] (t1) at (-0.8,1.2) {};
    \node[circle,draw,fill=lightgray] (t2) at (0,0.7) {};
    \node[circle,draw,fill=lightgray] (t3) at (0,-0.7) {};
    \node[circle,draw,fill=lightgray] (t4) at (-0.8,-1.2) {};

    \node at (0.2,0.3) {$\vdots$};
    \node at (0.2,-0.2) {$\vdots$};
    \draw[-] (d) -- (t1) -- (t2) -- (d);
    \draw[-] (d) -- (t3) -- (t4) -- (d);

    \draw[-][decorate,decoration={brace,amplitude=10pt}]
        ([xshift=30pt]t1.east) -- ([xshift=30pt]t4.east)
        node[midway,xshift=18pt] {$m$};
\end{tikzpicture}

 $(2m+2, 2^{2m+3})$ & 3 & $m+2$ \\
\hline
\end{tabular} 
 \caption{The four types of indecomposable non-split unigraphs mentioned in Theorem \ref{Tyshkevich-unsplittable},  their degree sequences, clique and independence numbers. }
 \label{UnsplittableTable}
 \end{center}
 \end{table}

 \newpage

\begin{table}[H]
\begin{center}
 \small
 \begin{tabular}{|>{\centering\arraybackslash}m{2.25in}|>{\centering\arraybackslash}m{3.5in}|} 
 \hline
{Indecomposable Split Unigraphs $G$} & Illustration and  Paired Degree Sequence  \\
 \hline
Single-vertex graph 
& 
\vspace*{1em}
\begin{tikzpicture}
    \node[circle,draw,fill=lightgray] (a1) at (0,1) {};
\end{tikzpicture} 

 $(0; \emptyset)$ or $(\emptyset; 0)$  \\ 
\hline
$S(p,q)$, $p \ge 1$, $q \ge 2$ & 
\vspace*{1em}

\begin{tikzpicture}  
 \node[circle,draw,fill=lightgray] (a4) at (4,1) {};
       \node[circle,draw,fill=lightgray] (d2) at (3.25,0) {};
       \draw[-] (a4) -- (d2);
       \node[circle,draw,fill=lightgray] (d3) at (4.25,0) {};
       \draw[-] (a4) -- (d3);
       \node[circle,draw,fill=lightgray] (d4) at (4.75,0) {};
       \draw[-] (a4) -- (d4);
       \node at (3.8,0) {$\cdots$};
      \draw[-][decorate,decoration={brace,mirror,amplitude=10pt}]
        ([yshift=-5pt]d2.south west) -- ([yshift=-5pt]d4.south east)
        node[midway,yshift=-15pt]{$p$};

 \node[circle,draw,fill=lightgray] (b4) at (6,1) {};
       \node[circle,draw,fill=lightgray] (e2) at (5.25,0) {};
       \draw[-] (b4) -- (e2);
       \node[circle,draw,fill=lightgray] (e3) at (6.25,0) {};
       \draw[-] (b4) -- (e3);
       \node[circle,draw,fill=lightgray] (e4) at (6.75,0) {};
       \draw[-] (b4) -- (e4);
       \node at (5.8,0) {$\cdots$};
      \draw[-][decorate,decoration={brace,mirror,amplitude=10pt}]
        ([yshift=-5pt]e2.south west) -- ([yshift=-5pt]e4.south east)
        node[midway,yshift=-15pt]{$p$};
        
  \draw[-] (a4) -- (b4);    
    \node at (7.25,0.5) {$\cdots$};   
    \node at (7.75, 0.5)  {$\cdots$};  
    
     \node[circle,draw,fill=lightgray] (c4) at (9,1) {};
       \node[circle,draw,fill=lightgray] (f2) at (8.25,0) {};
       \draw[-] (c4) -- (f2);
       \node[circle,draw,fill=lightgray] (f3) at (9.25,0) {};
       \draw[-] (c4) -- (f3);
       \node[circle,draw,fill=lightgray] (f4) at (9.75,0) {};
       \draw[-] (c4) -- (f4);
       \node at (8.8,0) {$\cdots$};
      \draw[-][decorate,decoration={brace,mirror,amplitude=10pt}]
        ([yshift=-5pt]f2.south west) -- ([yshift=-5pt]f4.south east)
        node[midway,yshift=-15pt]{$p$};
   \draw[-] (a4) to[out=25,in=155] (c4);
     \draw[-] (b4) to[out=15,in=165] (c4);
      \draw[-][decorate,decoration={brace, mirror, amplitude=10pt}]
        ([yshift=-25pt]d2.south west) -- ([yshift=-25pt]f4.south east)
node[midway,yshift=-16pt]{$q$};
\end{tikzpicture}

$((p+ q-1)^q; 1^{pq})$  \\

\hline
$S_2(p_1, q_1, \hdots, p_m, q_m)$,  \hspace*{0.03in} $m \ge 2$,  $p_1 > p_2 > \hdots > p_m \ge 1$ and  $q_i \ge 1$ for $i = 1, \hdots, m$
& 
\vspace{1em}
\begin{tikzpicture}
\fill[gray!30] 
        (0,0) -- (2,0) -- (1.5,1.5) -- (0.5,1.5) -- cycle;
    \draw[-] (0,0) -- (2,0) -- (1.5,1.5) -- (0.5,1.5) -- cycle;
     \node at (1,-0.5) {$S(p_1, q_1)$};
\fill[gray!30] 
        (2.25,0) -- (4.25,0) -- (3.75,1.5) -- (2.75,1.5) -- cycle;
 \draw[-] (2.25,0) -- (4.25,0) -- (3.75,1.5) -- (2.75,1.5) -- cycle;
 \node at (3.25,-0.5) {$S(p_2, q_2)$};   
 \node at (4.65,0.75) {$\cdots$};
  \node at (5.2,0.75) {$\cdots$};   
 \fill[gray!30] 
        (5.5,0) -- (7.5,0) -- (7.0,1.5) -- (6.0,1.5) -- cycle;
 \draw[-] (5.5,0) -- (7.5,0) -- (7.0,1.5) -- (6.0,1.5) -- cycle;
 \node at (6.5,-0.5) {$S(p_m, q_m)$};  
 
 \node (z1) at (1.0, 1.5) {};
 \node (z2) at (3.25, 1.5){};
 \node (z3) at (6.5, 1.5){};
 \draw[-] (z1) to[out=25,in=155] (z3);
 \draw[-] (z1) to[out=15,in=165] (z2);
  \draw[-] (z2) to[out=15,in=165] (z3);
 \end{tikzpicture}
     $((p_1+ N -1)^{q_1}, \cdots, $ $(p_m + N-1)^{q_m};$ $1^{p_1q_1 + \cdots + p_m q_m})$ where $N = \sum_{i=1}^m q_i$  \\ 
\hline
$S_3(p, q_1, q_2)$ with $p \ge 1, q_1 \ge 2, q_2 \ge 1$
&  \vspace{1em}
\begin{tikzpicture}
\fill[gray!30] 
        (0,0) -- (2,0) -- (1.5,1.5) -- (0.5,1.5) -- cycle;
    \draw[-] (0,0) -- (2,0) -- (1.5,1.5) -- (0.5,1.5) -- cycle;
     \node at (1,-0.5) {$S(p, q_1)$};
 \node[circle,draw,fill=lightgray]  (e) at (2.5,0) {$e$};
 \node (z1) at (0.6, 1.5) {};
 \node (z2) at (0.9, 1.5){};
 \node (z3) at (1.2, 1.5){};
 \draw[-] (e) -- (z1);
  \draw[-] (e) -- (z2);
  \draw[-] (e) -- (z3);
 \fill[gray!30] 
        (3,0) -- (5,0) -- (4.5,1.5) -- (3.5,1.5) -- cycle;
    \draw[-] (3,0) -- (5,0) -- (4.5,1.5) -- (3.5,1.5) -- cycle;
     \node at (4,-0.5) {$S(p+1, q_2)$};
   \node (z4) at (4, 1.5){};
   \draw[-] (z2) to[out=25,in=155] (z4);
\end{tikzpicture}

$((p + q_1 + q_2)^{q_1 + q_2}; q_1,$ $1^{pq_1 + (p+1)q_2})$ \\ 
\hline
$S_4(p,q)$ with $p \ge 1$, $q \ge 1$ & 
\vspace{1em}
\begin{tikzpicture}
\fill[gray!30] 
        (0,0) -- (2,0) -- (1.5,1.5) -- (0.5,1.5) -- cycle;
    \draw[-] (0,0) -- (2,0) -- (1.5,1.5) -- (0.5,1.5) -- cycle;
     \node at (1,-0.5) {$S(p,2)$};
 \node[circle,draw,fill=lightgray]  (e) at (2.5,0) {$e$};
  \node[circle,draw,fill=lightgray]  (f) at (2.5,1.5) {$f$};
 \node (z1) at (0.6, 1.4) {};
 \node (z2) at (0.9, 1.4){};
 \node (z3) at (1.2, 1.4){};
 \node (z4) at (0.3,0){};
 \node (z5) at (0.6,0){};
 \node (z6) at (0.9,0){};
 \node (z7) at (1.2,0){};
 \draw[-] (e) -- (z1);
  \draw[-] (e) -- (z2);
   \draw[-] (z1) to[out=25,in=155]   (f); 
  \draw[-] (z2) to[out=20,in=160]  (f);  
  \draw[-] (z3) to[out=15,in=165]  (f); 
   \draw[-] (z4) -- (f);  
  \draw[-] (z5) -- (f); 
   \draw[-] (z6) -- (f);  
  \draw[-] (z7) -- (f); 
  
 \fill[gray!30] 
        (3,0) -- (5,0) -- (4.5,1.5) -- (3.5,1.5) -- cycle;
    \draw[-] (3,0) -- (5,0) -- (4.5,1.5) -- (3.5,1.5) -- cycle;
     \node at (4,-0.5) {$S(p+1, q)$};
   \node (z4) at (4, 1.5){};
   \draw[-] (z2) to[out=35,in=145] (z4);
   \node (y1) at (3.9, 1.4) {};
 \node (y2) at (4.2, 1.4){};
 \node (y3) at (4.5, 1.4){};
 \node (y4) at (3.5,0){};
 \node (y5) at (3.8,0){};
 \node (y6) at (4.2,0){};
\node (y7) at (4.5,0){};
    \draw[-] (f) to[out=15,in=165]  (y1); 
  \draw[-] (f) to[out=20,in=160]  (y2);  
  \draw[-] (f) to[out=25,in=155]  (y3); 
    \draw[-] (y4) -- (f); 
  \draw[-] (y5) -- (f);  
  \draw[-] (y6) -- (f); 
   \draw[-] (y7) -- (f); 
\end{tikzpicture}

$(2(p+q+1)+qp, (p+q+3)^{q+2}; 2^{qp + 2p + q + 1})$  \\ 
\hline
\end{tabular} 
 \caption{The five types of indecomposable non-split graphs  mentioned in Theorem \ref{Tyshkevich-splittable} and their degree sequences. The $K$-parts of the split graphs are on top while the $S$-parts are at the bottom.}
 \label{SplittableTable}
 \end{center}
 \end{table}

\newpage

\section{Applications, Part 1}
\label{sec4}

We now  show  how Tyshkevich's work can be used to address algorithmic questions. It is well-known that given a graph $G$, computing the size of its largest clique $\omega(G)$, the size of its largest stable set $\alpha(G)$, the size of its smallest vertex cover $\beta(G)$ as well as its chromatic number $\chi(G)$ are NP-hard problems.  When $G$ is decomposable, can a decomposition simplify the computation of these parameters?   We begin with a nice result from Whitman \cite{Whitman24}.

\begin{lemma} {\bf \cite{Whitman24}}
Let $G= (G_k , A_k, B_k) \circ (G_{k-1}, A_{k-1}, B_{k-1})  \circ \hdots \circ (G_1, A_1, B_1) \circ G_0$. Then  $\omega(G) = |A_k| + |A_{k-1}|+ \hdots + |A_1| + \omega(G_0)$ and $\chi(G) = |A_k| + |A_{k-1}|+ \hdots + |A_1| + \chi(G_0)$.
\label{lemmaWhitman}
\end{lemma}

A similar result holds for $\alpha(G)$ and $\beta(G)$.

\begin{lemma}
Let $G= (G_k , A_k, B_k) \circ (G_{k-1}, A_{k-1}, B_{k-1})  \circ \hdots \circ (G_1, A_1, B_1) \circ G_0$. Then it is the case that $\alpha(G) = |B_k| + |B_{k-1}|+ \hdots + |B_1| + \alpha(G_0)$ and $\beta(G) = |A_k| + |A_{k-1}|+ \hdots + |A_1| + \beta(G_0)$.
\label{lemmaWhitman2}
\end{lemma}

\begin{proof}
We prove the formula for $\alpha(G)$ first using induction on the number of components of $G$.   When $G$ has only one component, $G = G_0$ so $\alpha(G) = \alpha(G_0)$.   Assume that the result holds for graphs that have $k$ or fewer components.  Let $G$ be a graph with $k+1$ components so $G= (G_k , A_k, B_k) \circ (G_{k-1}, A_{k-1}, B_{k-1})  \circ \hdots \circ (G_1, A_1, B_1) \circ G_0$. Let $H = (G_{k-1}, A_{k-1}, B_{k-1})  \circ \hdots \circ (G_1, A_1, B_1) \circ G_0$ so $G = (G_k, A_k, B_k) \circ H$.   Denote as $B_H$   a maximum-sized stable set in $H$.  Clearly, $B_k \cup B_H$ is a stable set of $G$ so $\alpha(G) \ge |B_k| + |B_H|$.  We will now show that $\alpha(G) \leq |B_k| + |B_H|$.

 Let $D$ be a stable set of $G$,  $D_k = D \cap V(G_k)$ and $D_H = D \cap V(H)$.  Notice that both $D_k$ and $D_H$ must be stable sets of $G_k$ and $H$ respectively.  Furthermore,  $D_k$ can contain at most one vertex in $A_k$ since the vertices in $A_k$ induce a clique.  Let us consider two cases:  (i) $|D_k \cap A_k| = 0$ and (ii) $|D_k \cap A_k| = 1$.  For case (i),  $|D| = |D_k| + |D_H| \leq |B_k| + |B_H|$.    For case (ii), since $D_k \cap A_k \neq \emptyset$,  $D_H$ has to be empty because  every vertex in $A_k$ is adjacent to every vertex of $H$ in $G$.  Thus,  $|D| = |D_k|  \leq |B_k| + 1  \leq |B_k| + |B_H|$.  The latter inequality holds because $H$ is a non-empty graph so $|B_H| \ge 1$.  We have now shown that for cases (i) and (ii), every stable set of $G$ has size at most $|B_k| + |B_H|$.  It follows that $\alpha(G) = |B_k| + |B_H|$.  But by the induction hypothesis, $|B_H| = |B_{k-1}|+ \hdots + |B_1| + \alpha(G_0)$ so $\alpha(G) = |B_k| + |B_{k-1}|+ \hdots + |B_1| + \alpha(G_0)$. 
 
 Now, $\alpha(G) + \beta(G) = |V(G)|$ and $|V(G)| = |V(G_k)| + |V(G_{k-1})| + \hdots +  |V(G_0)|$.   Thus, 
 \begin{eqnarray*}
 \beta(G) & = &  |V(G)| - \alpha(G) \\
                & = & (|V(G_k)| + |V(G_{k-1})| + \hdots + |V(G_1)| + |V(G_0)|) - (|B_k| + |B_{k-1}|+ \hdots + |B_1| + \alpha(G_0)) \\
                & = & (|V(G_k)| - |B_k|) + (|V(G_{k-1})| - |B_{k-1}|) +  \hdots + (|V(G_1)| - |B_1|) + (|V(G_0)| - \alpha(G_0)) \\
                & = &  |A_k| + |A_{k-1}| + \hdots + |A_1| + \beta(G_0)
 \end{eqnarray*}
 where the last equality uses the fact that $\alpha(G_0) + \beta(G_0) = |V(G_0)|$.
\end{proof}

We note that the lemmas above do not require that $G$ be expressed in terms of its canonical decomposition.  Some decomposition of $G$ will do.   We also note that it is almost the case that $\omega(G) = \sum_{i=0}^k \omega (G_i)$ or $\alpha(G) = \sum_{i=0}^k \alpha(G_i)$, etc.   Complications arise when the split graph components of $G$ are not balanced.  It may be the case that $\omega(G_i) = |A_i|+1$ or $\alpha(G_i) = |B_i| + 1$ so we cannot just replace $|A_i|$ with $\omega(G_i)$  nor $|B_i|$ with $\alpha(B_i)$ in the formulas of $\omega(G)$ and $\alpha(G)$ respectively.


\medskip	
\bigskip 
\centerline{$*  *  *  * *$}
\bigskip 
\medskip


To implement Lemmas \ref{lemmaWhitman} and \ref{lemmaWhitman2}, however, we shall decompose $G$ into its indecomposable components using {\tt Decompose($G$)}.  
It will output the paired degree sequence of each $(G_i, A_i, B_i)$, $i \ge 1$ in linear time.  Thus,   $|A_i|$ and $|B_i|$  for $i \ge 1$ can be obtained in linear time too.  This means that we just need $\omega(G_0)$, $\chi(G_0)$, $\alpha(G_0)$ and $\beta(G_0)$ to compute $\omega(G), \chi(G), \alpha(G)$ and $\beta(G)$ respectively.   In the case of unigraphs,  Tyshkevich has done the hard work of identifying all possible candidates for $G_0$.  

We shall make use of the next proposition in our proofs. 

  \begin{proposition} 
 When $G_i$ is an indecomposable split graph with two or more vertices, $G_i$ is balanced. That is, $G_i$ has a unique $KS$-partition $(A_i,B_i)$ so that $\omega(G_i) = |A_i|$ and $\alpha(G_i) = |B_i|$. 
 \label{indecomsplit}
  \end{proposition}

\begin{theorem}
When $G$ is a unigraph, the parameters $\omega(G), \chi(G)$, $\alpha(G)$ and $\beta(G)$ can all be computed in linear time.  
\end{theorem}

\begin{proof}
Let the canonical decomposition of $G$ be $(G_k , A_k, B_k) \circ \hdots \circ (G_1, A_1, B_1) \circ G_0.$   As noted above, {\tt Decompose($G$)} will output the paired degree sequence of each $G_i$, $i \ge 1$,  so  $|A_i|$ and $|B_i|$ for $i = 1, \hdots, k$ can be obtained in linear time.   Let us determine $\omega(G_0)$ and $\alpha(G_0)$ first.

Assume $G_0$  is a split graph.  When $G_0$ is a single-vertex graph,  $\omega(G_0) = \alpha(G_0) = 1$.  When $G_0$ has two or more vertices,  according to Proposition \ref{indecomsplit}, $G_0$ has a unique $KS$-partition $(A_0, B_0)$ with $\omega(G_0) = |A_0|$ and $\alpha(G_0) = |B_0|$.   Now, {\tt Decompose($G$)} will also output the degree sequence of $G_0$.  Using {\tt DetermineSplit} \cite{Ch25}, an algorithm that can recognize if a graph is split based on a theorem by Hammer and Simeone \cite{HaSi81}, the paired degree sequence of $(G_0, A_0, B_0)$, and therefore $|A_0|$ and $|B_0|$,  can also be obtained in linear time.  



 
 
 When $G_0$ is a non-split graph,  $G_0$  is isomorphic to one of the four types of graphs listed in Table \ref{UnsplittableTable} or their complements.  The clique and independence numbers of $C_5, mK_2, U_2(m, \ell)$ and $U_3(m)$ 
 are shown in the third and fourth columns of the table.  It is easy to verify that the values are correct. Additionally, the clique and independence numbers of their complements can be obtained from these values because $\omega(\overline{G}) = \alpha(G)$ while $\alpha(\overline{G}) = \omega(G)$.   Thus, to determine $\omega(G_0)$ and $\alpha(G_0)$, use {\tt IsUnigraph($G_0$)} to identify the type of $G_0$ and then use the information in Table \ref{UnsplittableTable} to obtain their corresponding clique and independence numbers.
 
 
 We have established that $\omega(G)$ and $\alpha(G)$ can be computed in linear time.  Since $\beta(G) = |V(G)| - \alpha(G)$, $\beta(G)$ can also be computed in linear time.  Finally,  Tyskevich proved in Corollary 11 of \cite{Ty00} that a unigraph $G$ is perfect if and only if $G_0$ is not a 5-cycle.   Recall that when a graph is perfect, its chromatic number is equal to its clique number.  Thus, when $G_0$ is not a 5-cycle, $\chi(G) = \omega(G)$.  When $G_0$ is a 5-cycle, $\chi(G) = \omega(G) +1$ according to Lemma \ref{lemmaWhitman} because $\chi(C_5) = 3$ but $\omega(C_5) = 2$.  Since $\omega(G)$ can be computed in linear time, so can $\chi(G)$. 
  \end{proof}


\section{Applications, Part 2}
\label{sec5}

Our initial interest in Tyshkevich's  decomposition theorem came about from our work on symmetry-breaking in graphs \cite{Ch06, ArChDe08, Ch09, Ch25}.  A symmetry of graph $G$ is an informal term for an {\it automorphism} of $G$, a bijection $\pi: V(G) \rightarrow V(G)$  that preserves the adjacencies of $G$.
For example,  $C_5$ has ten automorphisms:  five rotations and five reflections. We say an automorphism is {\it trivial} if it is the identity map; otherwise, it is {\it non-trivial}.  The set of automorphisms of $G$, $Aut(G)$,  form a group under the composition of functions operation so it is called the {\it automorphism group} of $G$.  

Suppose the canonical decomposition of $G$ is $(G_k , A_k, B_k)  \circ \hdots \circ (G_1, A_1, B_1) \circ G_0$.  We wanted to understand how an automorphism of $G$ behaves on the $G_i$'s.  
Might it be the case that $Aut(G)$ is isomorphic to $ Aut(G_k) \times Aut(G_{k-1}) \times \hdots \times Aut(G_1) \times Aut(G_0)?$
 That is, can every automorphism $\pi$ of $G$ be broken down to $k+1$ functions $(\pi_k, \pi_{k-1}, \hdots, \pi_1, \pi_0)$ so that each $\pi_i$ is an automorphism of $G_i$ and, conversely, the automorphisms of the $G_i$'s, $i = 0, \hdots, k$,  can be combined to create an automorphism of $G$?

 It turns out that the above conjecture is almost true.  It fails when $G$ has two or more consecutive indecomposable components that  are single-vertex graphs and they all have the same type.  Consider the threshold graph  in Figure \ref{fig3} whose indecomposable components are all single-vertex graphs.  Let $G_i$ denote the component containing  $v_i$.  Then $G_1$ and $G_2$ have type $S_1$ while $G_3$ and $G_4$ have type $K_1$.   And indeed the function $\pi(v_0) = v_0, \pi(v_1) = v_2, \pi(v_2) = v_1, \pi(v_3) = v_4, \pi(v_4) = v_3$ is an automorphism  of the graph.  It ``crosses" indecomposable components, mapping the vertices of $G_1$ and $G_2$ to each other and the vertices of $G_3$ and $G_4$ to each other. 

To address this issue, we proposed a {\it compact version} of Tyshkevich's decomposition theorem in \cite{Ch25}.   In the decomposition of $G$, when $G_i$, $i \ge 1$, is a single-vertex graph, it either has type $K_1$ or type $S_1$.  Additionally, when $G_0$ {\it and} $G_1$ are both single-vertex graphs, we shall say that the type of $G_0$ is the same as the type of $G_1$.  

The {\it compact canonical decomposition of $G$} is formed by {\it maximally} combining consecutive single-vertex components in 
the canonical decomposition of $G$ that have the same type into a single component.
That is, if there are $m$ consecutive components of type $K_1$, they are replaced with 
the complete graph on $m$ vertices with $K$-max as its $KS$-partition (i.e., the degree sequence is $((m-1)^m; \emptyset)$).  On the other hand, if their type is $S_1$, they are replaced with
the graph with $m$ isolated vertices whose $KS$-partition is $S$-max (i.e., the degree sequence is $(\emptyset; 0^m)$).


For example, if we describe the indecomposable components of $G$ using their degree sequences and 
$G = (\emptyset; 0) \circ (0; \emptyset) \circ (0; \emptyset) \circ (0; \emptyset) \circ (\emptyset; 0) \circ (\emptyset; 0) \circ (\emptyset; 0) \circ (0)$, then its compact canonical decomposition is $G = (\emptyset; 0) \circ(2^3; \emptyset) \circ (\emptyset; 0^4)$. On the other hand,  the canonical decomposition of the threshold graph in Figure \ref{fig3} is $(0; \emptyset) \circ (0; \emptyset) \circ (\emptyset; 0) \circ (\emptyset; 0) \circ (0)$ so its compact canonical decomposition is $(1^2;  \emptyset) \circ (\emptyset; 0^3)$.


\medskip

\begin{theorem}
 \label{canon2}
{\bf (The Compact Canonical Decomposition Theorem  \cite{Ch25})}   The compact canonical decomposition of a graph $G$ is obtained from Tyshkevich's canonical decomposition of $G$ by maximally combining the single-vertex components of the same type into a single component.  The compact decomposition of $G$ is unique up to isomorphism. Each component is an indecomposable graph with at least two vertices, a complete graph or a graph of isolated vertices.
\end{theorem}

We can now  characterize the automorphisms of $G$. 
\medskip

\begin{theorem}
Let $G= (G'_\ell , A'_\ell, B'_\ell) \circ (G'_{\ell-1}, A'_{\ell-1}, B'_{\ell-1})  \circ \hdots \circ (G'_1, A'_1, B'_1) \circ G'_0$ be the compact canonical decomposition of $G$.  Then $Aut(G)$ is isomorphic to 
$$ Aut(G'_\ell) \times Aut(G'_{\ell-1})   \times \hdots \times Aut(G'_1) \times Aut(G'_0).$$
\label{thmautomorph}
\end{theorem}


\centerline{$*  *  *  * *$}
\bigskip 
\medskip

Our results in \cite{Ch25} imply Theorem \ref{thmautomorph},  but we prove it here for completeness.  When $G$ is a split graph with $KS$-partition $(A,B)$, let $Aut(G,A,B)$ contain the automorphisms of $G$ that fixes $A$ and $B$ too; i.e., the automorphism maps the vertices in $A$ to vertices in $A$ and the vertices in $B$ to vertices in $B$.  
We shall make use of the results below.  

\begin{proposition}
{\bf  \cite{Ch25}}
   When $G$ is a balanced split graph with $KS$-partition $(A,B)$,  $Aut(G) = Aut(G, A, B)$. 
   \label{balanced}
  \end{proposition}

\smallskip

\begin{proposition} {\bf  \cite{Ch25}}
  \label{decomposable}  Let $G$ be a graph with two or more vertices. 

\noindent (i) If $G$ has an isolated vertex $u$ then $G = (\{u\}, \emptyset, \{u\}) \circ (G - u)$. 
  

\noindent (ii) If $G$ has a dominant\footnote{A vertex is {\it dominant} if it is adjacent to all other vertices in the graph.
}  vertex $v$  then $G = (\{v\}, \{v\}, \emptyset) \circ (G-v)$. 
   
 
 
\end{proposition}

\smallskip

\begin{lemma}
 \label{autchar3}
 {\bf  \cite{Ch25}} Let $G^* = (G,A,B) \circ H$.  There is an automorphism $\pi$ of $G^*$ that maps $u \in V(G)$ to $v \in V(H)$ if and only if one of the following conditions hold:
 
 \noindent (i) $u \in A$ and a swing vertex of $G$ while $v$ is a dominant vertex of $H$ or 
 
 \noindent (ii) $u \in B$ and a swing vertex of $G$ while $v$ is an isolated vertex of $H$.
 
  


\end{lemma}

Lemma \ref{autchar3} is important because it describes the two conditions when an automorphism of $G^*$ crosses over, mapping a vertex of $G$ to a vertex of $H$.  Let us now prove Theorem \ref{thmautomorph}.

\begin{proof}
We shall do induction on $\ell$.  When $\ell = 0$,  $G = G'_0$ so $Aut(G) = Aut(G'_0)$ is true.  Assume the theorem holds for graphs with $\ell$ or fewer components in its compact canonical decomposition.  Let $G$ be a graph with $\ell + 1$ components in its compact canonical decomposition.  Thus, we can express $G$ as  $G =  (G'_\ell, A'_\ell, B'_\ell) \circ H$ where $H = (G'_{\ell-1}, A'_{\ell-1}, B'_{\ell-1}) \circ \cdots \circ (G'_1, A'_1, B'_1) \circ G'_0$.   We shall argue that $Aut(G) = Aut(G'_\ell) \times Aut(H)$.

There are three possibilities for $G'_\ell$ according to Theorem \ref{canon2}.  For each case, we will argue that the two conditions of Lemma \ref{autchar3} do not hold.  It will then imply that  every automorphism $\pi$ of $G$ maps $V(G'_\ell)$ to $V(G'_\ell)$ and $V(H)$ to $V(H)$.  That is, $Aut(G) \subseteq Aut(G'_\ell) \times Aut(H)$.

When $G'_\ell$ is an indecomposable split graph with at least two vertices,  $G'_\ell$ is a {\it balanced} split graph by Proposition \ref{indecomsplit} and therefore has no swing vertices.  Conditions (i) and (ii) in Lemma \ref{autchar3} do not hold. 


Suppose instead that $G'_\ell$ is a complete graph with $KS$-partition $((m-1)^m; \emptyset)$.   Notice that condition (ii) of Lemma \ref{autchar3} does not hold because $B$ is an empty set, so we only have to check condition (i).  $G'_\ell$ is a type-$K_1$ single-vertex graph or was formed by maximally combining consecutive type-$K_1$ single-vertex graphs.   If $|H| = 1$,  then $H = G'_0$ is a single vertex graph that should have been combined with $G'_\ell$ in the compact canonical decomposition of $G$.  Thus, $|H| \ge 2$.  If $H$ has a dominant vertex, then $H$ is decomposable by Proposition \ref{decomposable} with the leftmost component of type $K_1$.  Again, this type-$K_1$ component should have been combined with $G'_\ell$.  Hence, $H$ has no dominant vertices.  



Finally, assume $G'_\ell$ is a graph of isolated vertices with $KS$-partition $(\emptyset; (m-1)^m)$.  
Condition (i) of Lemma \ref{autchar3} does not hold because $A$ is an empty set.  When $|H| = 1$, $H = G'_0$ and should have been combined with $G'_\ell$.  So $|H| \ge 2$.   If $H$ has an isolated vertex, then $H$ is decomposable by Proposition \ref{decomposable} with the leftmost component of type $S_1$.  Again, the leftmost component should have been combined with $G'_\ell$.  Thus, $H$ has no isolated vertices so condition (ii) of Lemma \ref{autchar3} does not hold.

We have established that $Aut(G) \subseteq Aut(G'_\ell) \times Aut(H)$.  Let us prove the converse.  Let $\pi_1 \in Aut(G'_\ell)$ and $\pi_2 \in Aut(H)$.   Define $\pi$ on $V(G)$ so that $\pi(v) = \pi_1(v)$ when $v \in V(G'_\ell)$ and $\pi(v) = \pi_2(v)$ when $v \in V(H)$.  We now have to argue that $\pi$ preserves the adjacencies of $G$.  

Let $x, y \in V(G'_\ell)$ or $x, y \in V(H)$.  Then $xy$ is an edge of $G$ if and only if $\pi(x) \pi(y)$ is an edge of $G$ because $\pi$ behaves like $\pi_1$ in $G'_\ell$ and like $\pi_2$ in $H$. 
 Thus, assume $x \in V(G'_\ell)$ and $y \in V(H)$; the case when $x \in V(H)$ and $y \in V(G'_\ell)$ is addressed similarly.   Suppose $xy$ is an edge of $G$.  Since none of the vertices in $B'_\ell$ are adjacent to the vertices in $H$,  $x \in A'_\ell$.   If we can show that $\pi(x) \in A'_\ell$, then $\pi(x) \pi(y)$ is an edge of $G$ too because every vertex in $A'_\ell$ is adjacent to every vertex in $H$.  

Since $A'_\ell \neq \emptyset$, $G'_\ell$ cannot be a graph of isolated vertices.   When  $G'_\ell$  is an  indecomposable graph with at least two vertices, $(A'_\ell, B'_\ell)$ is its unique $KS$-partition.  Every automorphism of $G'_\ell$  must fix $A'_\ell$ and $B'_\ell$ so $\pi_1(x) \in A'_\ell$ because $x \in A'_\ell$.  On the other hand, when $G'_\ell$ is a complete graph  with $KS$-partition $((m-1)^m; \emptyset)$, all of its vertices are in $A'_\ell$ so $\pi_1(x) \in A'_\ell$ too.  We have shown that when $x \in A'_\ell$, $\pi(x) = \pi_1(x) \in A'_\ell$.

 On the other hand, suppose $xy$ is not an edge of $G$.  Then $x \in B'_\ell$.   If we can show that $\pi(x) \in B'_\ell$ too, then $\pi(x) \pi(y)$ is also not an edge of $G$.  Now, $G'_\ell$ cannot be a complete graph because $B'_\ell \neq \emptyset$.  So $G'_\ell$ must be an  indecomposable graph with at least two vertices with $(A'_\ell, B'_\ell)$ as its unique $KS$-partition or $G'_\ell$ is a graph of isolated vertices with $KS$-partition $(\emptyset; (m-1)^m)$.  In the former case, $\pi_1$ fixes $B'_\ell$ so $\pi_1(x) \in B'_\ell$.  In the latter case, all the vertices of $G'_\ell$ are in $B'_\ell$ so $\pi_1(x) \in B'_\ell$ trivially.   Thus, $\pi(x) = \pi_1(x) \in B'_\ell$. 
 
 We have shown that every automorphism in $Aut(G'_\ell) \times Aut(H)$ is also an automorphism of $Aut(G)$ so $Aut(G) = Aut(G'_\ell) \times Aut(H)$.  By the induction hypothesis,  $Aut(H) = Aut(G'_{\ell-1}) \times \cdots \times Aut(G'_1) \times Aut(G'_0)$ and therefore $Aut(G) = Aut(G'_\ell) \times Aut(G'_{\ell-1}) \times \cdots \times Aut(G'_1) \times Aut(G'_0).$
\end{proof}

\bigskip 
\centerline{$*  *  *  * *$}
\bigskip 
\medskip

In graph theory, the objective of symmetry-breaking is to tell the vertices of a graph apart.  Two vertices $u$ and $v$ can be confused with each other if there is an automorphism that maps $u$ to $v$.  One way to remove the confusion is to color the vertices, like what we did in   Figure \ref{fig5} for $C_5$.  
The vertices that are in danger of being confused with each other are the three red ones.  But one red vertex has  neighbors colored blue and green, another has neighbors colored red and green,  and yet another has neighbors colored red and blue.  Thus, all the vertices can be identified based on their color and that of their neighbors.


\begin{figure}[t]
\begin{center}
\begin{tikzpicture}
    \node[circle,draw,fill=red] (v1) at (90:1) {};    
    \node[circle,draw,fill=red] (v2) at (162:1) {};   
    \node[circle,draw,fill=green] (v3) at (234:1) {};   
    \node[circle,draw,fill=red] (v4) at (306:1) {};   
    \node[circle,draw,fill=blue] (v5) at (18:1) {};  
    \draw[-] (v1) -- (v2);
    \draw[-] (v2) -- (v3);
    \draw[-] (v3) -- (v4);
    \draw[-] (v4) -- (v5);
    \draw[-] (v5) -- (v1);
\end{tikzpicture}
\hspace*{3em}
\begin{tikzpicture}
    \node[circle,draw] (v1) at (90:1) {};    
    \node[circle,draw,fill=red] (v2) at (162:1) {};   
    \node[circle,draw] (v3) at (234:1) {};   
    \node[circle,draw] (v4) at (306:1) {};   
    \node[circle,draw,fill=blue] (v5) at (18:1) {};  
    \draw[-] (v1) -- (v2);
    \draw[-] (v2) -- (v3);
    \draw[-] (v3) -- (v4);
    \draw[-] (v4) -- (v5);
    \draw[-] (v5) -- (v1);
\end{tikzpicture}
\end{center}
\label{fig5}
\caption{A distinguishing labeling of $C_5$ on the left and a fixing set of $C_5$ on the right. Vertices that are part of the fixing set are assigned distinct colors and those not in the set are assigned a null color.}
\end{figure}
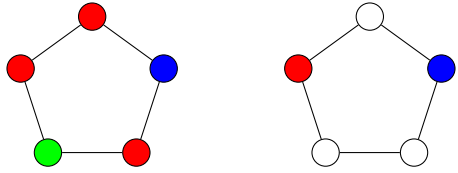



 
In {\it distinguishing labelings}, the vertices of a graph are assigned colors so that the labeled version of the graph has only one automorphism, the identity map. Formally, a {\it distinguishing labeling $\phi$} of $G$ \cite{AlCo96} colors the vertices of $G$ so that for every non-trivial automorphism $\pi$ of $G$, $\phi(v) \neq \phi(\pi(v))$ for some vertex $v$.  In other words, $\phi$ {\it breaks} all the non-trivial automorphisms of $G$ so that the only automorphism of the labeled graph $(G, \phi)$ is the identity map. The {\it distinguishing number} of $G$, $D(G)$, is the fewest number of colors needed to create a distinguishing labeling of $G$.   For example, $D(K_n) = n$, $D(P_n) = 2$ for $n \ge 2$ and $D(K_{n,n}) = n+1$.

 A concept related to distinguishing labelings is {\it fixing sets} \cite{DBLP:journals/dm/ErwinH06, DBLP:journals/combinatorics/Boutin06, DBLP:journals/gc/FijavzM10}.  A subset $S$ of $V(G)$ is a {\it fixing set} of $G$  if the only automorphism of $G$ that fixes every element in $S$ is the identity map.  That is, whenever an automorphism of $G$ maps each $s \in S$ to itself, it also maps each $v \not \in S$ to itself.  Figure \ref{fig5} shows a fixing set of $C_5$.   The {\it fixing number} of $G$, $Fix(G)$, is the size of a smallest fixing set.   For example,  $Fix(K_n) = n-1$, $Fix(P_n) = 1$ for $n \ge 2$ while $Fix(K_{n,n} )= 2(n-1)$.     We note that $Fix(G) = 0$ if and only if $G$ is a {\it rigid} graph (i.e., $G$ has no non-trivial automorphisms).
 
Every fixing set $S$ of $G$ can be translated into a distinguishing labeling $\phi_S$ by {\it individualizing} each vertex in $S$.  Each vertex in $S$ is assigned a distinct color while the vertices not in $S$ are assigned another color (e.g., the ``null" color).   Thus, $D(G) \leq Fix(G) +1$.  Sometimes, $D(G) = Fix(G) + 1$ but $D(G)$ can be significantly less than $Fix(G)$  because  colors can be reused in distinguishing labelings but not in fixing sets with the exception of the null color.

 One of our research questions was whether $D(G)$ can be computed efficiently when $G$ is a unigraph.  This line of inquiry was intriguing  because unigraphs have a very simple isomorphism algorithm:  when $G$ is a unigraph and $H$ is another graph, $G$ and $H$ are isomorphic if and only if $G$ and $H$ have identical degree sequences.  In our work on distinguishing labelings \cite{Ch06, ArChDe08, Ch09}, we found that whenever $G$ is part of a graph family that has an efficient isomorphism algorithm,  $D(G)$ can be computed in polynomial time.  Tyshkevich's results were essential in addressing the problem.
  
Theorem \ref{thmautomorph} makes symmetry-breaking straightforward. Informally, it implies that if we wish
to break the non-trivial automorphisms of $G$, all we have to do is break the non-trivial automorphisms of the individual
components in the compact canonical decomposition of $G$.



\begin{theorem} {\bf \cite{Ch25}}
Let $G= (G'_\ell , A'_\ell, B'_\ell) \circ (G'_{\ell-1}, A'_{\ell-1}, B'_{\ell-1})  \circ \hdots \circ (G'_1, A'_1, B'_1) \circ G'_0$ be the compact canonical decomposition of $G$.  Let $\phi$ be a labeling of $G$ and denote as $\phi_i$ the labeling on $G'_i$ when $\phi$ is restricted to $V(G'_i)$.   Then $\phi$ is a distinguishing labeling of $G$ if and only if $\phi_i$ is a distinguishing labeling of $G'_i$ for $i = 0, 1, \hdots, \ell$.  Consequently, $D(G) = \max \{D(G'_i), i = 0, \hdots, \ell\}$.
\label{distthm}
\end{theorem}

  To understand why the theorem holds, suppose $\phi_i$ is not distinguishing for some $G'_i$.  Then $(G'_i, \phi_i)$ has a non-trivial automorphism $\pi_i$.   When the latter is  combined with the identity maps of $G'_j$, $j \neq i$,  the result is a non-trivial automorphism of $(G, \phi)$.  Thus, $\phi$ is not distinguishing for $G$. 
  On the other hand, suppose $\phi$ is not a distinguishing labeling of $G$ so $(G, \phi)$ has a non-trivial automorphism $\pi$.  Then $\pi$ can be expressed as $(\pi_\ell, \pi_{\ell-1}, \hdots, \pi_1, \pi_0)$ where each $\pi_i$ is an automorphism of $(G'_i, \phi_i)$.   One of them, say $\pi_j$,   is a non-trivial automorphism of $(G'_j, \phi_j)$ since $\pi$ is non-trivial. Thus, $\phi_j$ is not a distinguishing labeling of $G'_j$. Finally, the colors used to create a distinguishing labeling for $G'_i$ can be reused to create distinguishing labelings for $G'_j$, $j \neq i$.  Hence,  $D(G) = \max \{D(G'_i), i = 0, \hdots, \ell\}$.


Applying the same reasoning to fixing sets, we obtain the next result.

\begin{theorem} 
Let $G= (G'_\ell , A'_\ell, B'_\ell) \circ (G'_{\ell-1}, A'_{\ell-1}, B'_{\ell-1})  \circ \hdots \circ (G'_1, A'_1, B'_1) \circ G'_0$ be the compact canonical decomposition of $G$.  Let $S \subseteq V(G)$ and let $S_i = S \cap V(G'_i)$.  Then $S$ is a fixing set of $G$ if and only if $S_i$ is a fixing set of $G'_i$    for $i = 0, 1, \hdots, \ell$.  Consequently, $Fix(G) = \sum_{i=0}^\ell Fix(G'_i)$.  
\label{fixthm}
\end{theorem}


\medskip	
\bigskip 
\centerline{$*  *  *  * *$}
\bigskip 
\medskip

To compute the distinguishing and fixing numbers of unigraphs,  we again make use of Tyshkevich's characterization of the indecomposable unigraphs.  We refer readers to \cite{Ch25} where we show that  the distinguishing numbers of unigraphs can be computed in linear time.   


\begin{theorem}
{\bf \cite{Ch25}}
When $G$ is a unigraph, $D(G)$ can be computed in linear time.
\end{theorem}

Let us now prove a similar result for the fixing numbers of unigraphs.   We start with some basic results. We then use them to determine the fixing numbers of the graphs in Tables \ref{UnsplittableTable} and \ref{SplittableTable}.


\begin{proposition}
\label{propbasic}
The following are true: 

\noindent (i) $Fix(C_n) = 2$ for $n \ge 3$, $Fix(P_n) = 1$ for $n \ge 2$,  $Fix(mK_2) = m$, $Fix(K_{1, \ell}) = \ell - 1$ when $\ell \ge 2$, $Fix(K_n) = n-1$. 

\noindent (ii) When $G$ is a rigid graph, $Fix(mG) = m-1$. 

\noindent (iii) When $G$ and $H$ are both non-rigid graphs, $Fix(G \cup H) = Fix(G) + Fix(H)$.
\end{proposition}

\smallskip
\begin{lemma}
The fixing numbers of the graphs  in Table \ref{UnsplittableTable2} are correct. 
\end{lemma}

\begin{proof}
The fixing numbers for $C_5$, $mK_2$ and $U_2(m, \ell) = mK_2 \cup K_{1, \ell}$ follow directly from Proposition \ref{propbasic}.  For $U_3(m)$, the center vertex has degree at least 4 while all other vertices have degree 2 so   every automorphism of $U_3(m)$ must map the center vertex to itself. Thus, a fixing set for $U_3(m)$ is effectively fixing $mK_2 \cup P_3$. According to Proposition \ref{propbasic}, $Fix(mK_2 \cup P_3) = Fix(mK_2) + Fix(P_3) = m + 1$ so $Fix(U_3(m)) = m+1$. 
\end{proof}

\begin{table}[t]
 \begin{center}
\small
 \begin{tabular}{|>{\centering\arraybackslash}m{1.5in}|>{\centering\arraybackslash}m{2.5in}|>{\centering\arraybackslash}m{1in}|}
 \hline
{Indecomposable Non-split Unigraphs $G$} & Minimum-sized fixing set &  ${Fix(G)}$  \\
\hline
$C_5$ & 
\vspace{1em}
\begin{tikzpicture}
    \node[circle,draw,fill=white] (v1) at (90:1) {};    
    \node[circle,draw,fill=gray] (v2) at (162:1) {};   
    \node[circle,draw,fill=white] (v3) at (234:1) {};   
    \node[circle,draw,fill=white] (v4) at (306:1) {};   
    \node[circle,draw,fill=gray] (v5) at (18:1) {};  
    \draw[-] (v1) -- (v2);
    \draw[-] (v2) -- (v3);
    \draw[-] (v3) -- (v4);
    \draw[-] (v4) -- (v5);
    \draw[-] (v5) -- (v1);
\end{tikzpicture}
  & 2 \\
\hline
$mK_2$ with $m \ge 2$ &

\vspace{1em}
\begin{tikzpicture}
    \node[circle,draw,fill=white] (a1) at (0,1) {};
    \node[circle,draw,fill=gray] (b1) at (0,0) {};
    \draw[-] (a1) -- (b1);
    
    \node[circle,draw,fill=white] (a2) at (0.75,1) {};
    \node[circle,draw,fill=gray] (b2) at (0.75,0) {};
    \draw[-] (a2) -- (b2);
    
    \node at (1.5,0.5) {$\cdots$};
    \node[circle,draw,fill=white] (a3) at (2.25,1) {};
    \node[circle,draw,fill=gray] (b3) at (2.25,0) {};
    \draw[-] (a3) -- (b3);
    
    \draw[-][decorate,decoration={brace,mirror,amplitude=10pt}]
        ([yshift=-5pt]b1.south west) -- ([yshift=-5pt]b3.south east)
        node[midway,yshift=-15pt]{$m$};
\end{tikzpicture} &  $m$ \\
\hline
$U_2(m,\ell) = mK_2 \cup K_{1,\ell}$   with $m \ge 1, \ell \ge 2$
& 
\vspace{1em}
\begin{tikzpicture}
    \node[circle,draw,fill=white] (a1) at (0,1) {};
    \node[circle,draw,fill=gray] (b1) at (0,0) {};
    \draw[-] (a1) -- (b1);
    \node[circle,draw,fill=white] (a2) at (0.75,1) {};
    \node[circle,draw,fill=gray] (b2) at (0.75,0) {};
    \draw[-] (a2) -- (b2);
    \node at (1.4,0.5) {$\cdots$};
    \node[circle,draw,fill=white] (a3) at (2.0,1) {};
    \node[circle,draw,fill=gray] (b3) at (2.0,0) {};
    \draw[-] (a3) -- (b3);   
    \draw[-][decorate,decoration={brace,mirror,amplitude=10pt}]
        ([yshift=-5pt]b1.south west) -- ([yshift=-5pt]b3.south east)
        node[midway,yshift=-15pt]{$m$};
      \node[circle,draw,fill=white] (a4) at (4,1) {};
      \node[circle,draw,fill=gray] (d1) at (3,0) {};
      \draw[-] (a4) -- (d1);
       \node[circle,draw,fill=gray] (d2) at (3.5,0) {};
       \draw[-] (a4) -- (d2);
       \node[circle,draw,fill=gray] (d3) at (4.5,0) {};
       \draw[-] (a4) -- (d3);
       \node[circle,draw,fill=white] (d4) at (5.0,0) {};
       \draw[-] (a4) -- (d4);
       \node at (4.05,0) {$\cdots$};
      \draw[-][decorate,decoration={brace,mirror,amplitude=10pt}]
        ([yshift=-5pt]d1.south west) -- ([yshift=-5pt]d4.south east)
        node[midway,yshift=-15pt]{$\ell$};
\end{tikzpicture}
 &  $m+ \ell -1$  \\
\hline
$U_3(m)$, $m \ge 1$ &  
\vspace{1em}
\begin{tikzpicture}[scale=1]

    \node[circle,draw,fill= white] (a) at (-2,1) {};
    \node[circle,draw,fill=white] (b) at (-3,0) {};
    \node[circle,draw,fill=gray] (c) at (-2,-1) {};
    \node[circle,draw,fill=white] (d) at (-1,0) {};
    
    \draw[-] (a) -- (b) -- (c) -- (d) -- (a);


    \node[circle,draw,fill=gray] (t1) at (-0.8,1.2) {};
    \node[circle,draw,fill=white] (t2) at (0,0.7) {};
    \node[circle,draw,fill=gray] (t3) at (0,-0.7) {};
    \node[circle,draw,fill=white] (t4) at (-0.8,-1.2) {};

    \node at (0.2,0.3) {$\vdots$};
    \node at (0.2,-0.2) {$\vdots$};
    \draw[-] (d) -- (t1) -- (t2) -- (d);
    \draw[-] (d) -- (t3) -- (t4) -- (d);

    \draw[-][decorate,decoration={brace,amplitude=10pt}]
        ([xshift=30pt]t1.east) -- ([xshift=30pt]t4.east)
        node[midway,xshift=18pt] {$m$};

\end{tikzpicture}
& $m+1$  \\
\hline
\end{tabular} 
 \caption{The four types of indecomposable non-split unigraphs, their minimum-sized fixing sets marked in gray and their fixing numbers.}
 \label{UnsplittableTable2}
 \end{center}
 \end{table}

\begin{lemma}
 The fixing numbers of  the graphs in Table \ref{SplittableTable2} are correct. 
 \end{lemma}

\begin{proof}
Let $K^*_{1, p}$ be $K_{1,p}$ rooted at its center.\footnote{Note that when $p \ge 2$. $Aut(K^*_{1,p}) = Aut(K_{1, p})$.  Thus, we only really need $K^*_{1,p}$ when $p = 1$.} Every automorphism of $S(p,q)$ maps a center of $K_{1,p}$ to itself or to the other centers.  Thus, we can simply think of $S(p,q)$ as $qK^*_{1, p}$.  When $p = 1$,  $K^*_{1, p}$ is a rigid graph so  $Fix(qK^*_{1,p}) = q-1$ by Proposition \ref{propbasic}.  But when $p > 1$, $K^*_{1, p}$ is just $K_{1,p}$ and no longer a rigid graph.  Thus, $Fix(qK^*_{1,p}) =  q Fix(K_{1,p})= q(p-1)$ by Proposition \ref{propbasic}.

For $S_2(p_1, q_1, p_2, q_2, \cdots, p_m, q_m)$,  an automorphism of the graph will only map the center of one star to the center of another star if and only if  they have the same number of leaves.  Thus, we can just focus on fixing each $S(p_i, q_i)$ for $i = 1, \hdots, m$.   For $S_3(p, q_1, q_2)$,  every automorphism of the graph will map $e$ to itself so we can just fix the subgraph $S(p, q_1) \cup S(p+1, q_2)$. Finally, for $S_4(p,q)$, every automorphism of the graph will fix $e$ and $f$ so we can just fix $S(p,2)$ and $S(p+1,q)$.   Their fixing numbers follow from our discussion. 
\end{proof}

Now that we know the fixing numbers of the indecomposable unigraphs in Tables \ref{UnsplittableTable} and \ref{SplittableTable},  we need to also consider the fixing numbers of their complements, inverses and complements of their inverses.

\begin{table}
 \small
\label{Splittable}
 \begin{tabular}{|>{\centering\arraybackslash}m{1.8in}|>{\centering\arraybackslash}m{3.2in}|>{\centering\arraybackslash}m{1.3in}|}
 \hline
{Indecomposable Split Unigraphs $G$ with $|V(G)| \ge 2$} & Minimum-sized fixing set &  ${Fix(G)}$  \\
\hline
$S(p,q)$, $p \ge 1$, $q \ge 2$ & 
\vspace*{1em}

\begin{tikzpicture}
   \node[circle,draw,fill=white] (a1) at (0,1) {};
    \node[circle,draw,fill=gray] (b1) at (0,0) {};
    \draw[-] (a1) -- (b1);
    \node[circle,draw,fill=white] (a2) at (0.75,1) {};
    \node[circle,draw,fill=gray] (b2) at (0.75,0) {};
    \draw[-] (a2) -- (b2);
    \node at (1.5,0.5) {$\cdots$};
    \node[circle,draw,fill=white] (a3) at (2.25,1) {};
    \node[circle,draw,fill=gray] (b3) at (2.25,0) {};
    \draw[-] (a3) -- (b3);
    \node[circle,draw,fill=white] (a4) at (3,1) {};
    \node[circle,draw,fill=white] (b4) at (3,0) {};
    \draw[-] (a4) -- (b4);
    
    \draw[-] (a1) -- (a2);
    \draw[-] (a3) -- (a4);
    \draw[dashed,-] (a2) -- (a3);
      \draw[-] (a1) to[out=35,in=145] (a4);
       \draw[-] (a1) to[out=35,in=145] (a3);
     \draw[-] (a2) to[out=25,in=155] (a4); 
    \draw[-][decorate,decoration={brace,mirror,amplitude=10pt}]
        ([yshift=-5pt]b1.south west) -- ([yshift=-5pt]b4.south east)
        node[midway,yshift=-15pt]{$q$};
\end{tikzpicture}
\begin{tikzpicture}  
 \node[circle,draw,fill=white] (a4) at (4,1) {};
       \node[circle,draw,fill=gray] (d2) at (3.25,0) {};
       \draw[-](a4) -- (d2);
       \node[circle,draw,fill=gray] (d3) at (4.25,0) {};
       \draw[-](a4) -- (d3);
       \node[circle,draw,fill=white] (d4) at (4.75,0) {};
       \draw[-](a4) -- (d4);
       \node at (3.8,0) {$\cdots$};
      \draw[-][decorate,decoration={brace,mirror,amplitude=10pt}]
        ([yshift=-5pt]d2.south west) -- ([yshift=-5pt]d4.south east)
        node[midway,yshift=-15pt]{$p$};

 \node[circle,draw,fill=white] (b4) at (6,1) {};
       \node[circle,draw,fill=gray] (e2) at (5.25,0) {};
       \draw[-](b4) -- (e2);
       \node[circle,draw,fill=gray] (e3) at (6.25,0) {};
       \draw[-](b4) -- (e3);
       \node[circle,draw,fill=white] (e4) at (6.75,0) {};
       \draw[-](b4) -- (e4);
       \node at (5.8,0) {$\cdots$};
      \draw[-][decorate,decoration={brace,mirror,amplitude=10pt}]
        ([yshift=-5pt]e2.south west) -- ([yshift=-5pt]e4.south east)
        node[midway,yshift=-15pt]{$p$};
        
  \draw[-](a4) -- (b4);    
    \node at (7.25,0.5) {$\cdots$};   
    \node at (7.75, 0.5)  {$\cdots$};  
    
     \node[circle,draw,fill=white] (c4) at (9,1) {};
       \node[circle,draw,fill=gray] (f2) at (8.25,0) {};
       \draw[-](c4) -- (f2);
       \node[circle,draw,fill=gray] (f3) at (9.25,0) {};
       \draw[-](c4) -- (f3);
       \node[circle,draw,fill=white] (f4) at (9.75,0) {};
       \draw[-](c4) -- (f4);
       \node at (8.8,0) {$\cdots$};
      \draw[-][decorate,decoration={brace,mirror,amplitude=10pt}]
        ([yshift=-5pt]f2.south west) -- ([yshift=-5pt]f4.south east)
        node[midway,yshift=-15pt]{$p$};
   \draw[-](a4) to[out=25,in=155] (c4);
     \draw[-](b4) to[out=15,in=165] (c4);
      \draw[-][decorate,decoration={brace, mirror, amplitude=10pt}]
        ([yshift=-25pt]d2.south west) -- ([yshift=-25pt]f4.south east)
node[midway,yshift=-16pt]{$q$};
\end{tikzpicture}
& $q-1$ when $p=1$;  
\vspace*{1in}

$q(p-1)$  otherwise\\
\hline
$S_2(p_1, q_1, \hdots, p_m, q_m)$,  \hspace*{0.03in} $m \ge 2$ and $p_1 > p_2 > \hdots > p_m \ge 1$ & 
\vspace{1em}
\begin{tikzpicture}
\fill[gray!30] 
        (0,0) -- (2,0) -- (1.5,1.5) -- (0.5,1.5) -- cycle;
    \draw[-](0,0) -- (2,0) -- (1.5,1.5) -- (0.5,1.5) -- cycle;
     \node at (1,-0.5) {$S(p_1, q_1)$};
\fill[gray!30] 
        (2.25,0) -- (4.25,0) -- (3.75,1.5) -- (2.75,1.5) -- cycle;
 \draw[-](2.25,0) -- (4.25,0) -- (3.75,1.5) -- (2.75,1.5) -- cycle;
 \node at (3.25,-0.5) {$S(p_2, q_2)$};   
 \node at (4.65,0.75) {$\cdots$};
  \node at (5.2,0.75) {$\cdots$};   
 \fill[gray!30] 
        (5.5,0) -- (7.5,0) -- (7.0,1.5) -- (6.0,1.5) -- cycle;
 \draw[-](5.5,0) -- (7.5,0) -- (7.0,1.5) -- (6.0,1.5) -- cycle;
 \node at (6.5,-0.5) {$S(p_m, q_m)$};  
 
 \node (z1) at (1.0, 1.5) {};
 \node (z2) at (3.25, 1.5){};
 \node (z3) at (6.5, 1.5){};
 \draw[-](z1) to[out=25,in=155] (z3);
 \draw[-](z1) to[out=15,in=165] (z2);
  \draw[-](z2) to[out=15,in=165] (z3);
 \end{tikzpicture}
     & $\sum_{i=1}^m Fix(S(p_i, q_i))$ \\
\hline
$S_3(p, q_1, q_2)$ with $p \ge 1, q_1 \ge 2, q_2 \ge 1$
&  \vspace{1em}
\begin{tikzpicture}
\fill[gray!30] 
        (0,0) -- (2,0) -- (1.5,1.5) -- (0.5,1.5) -- cycle;
    \draw[-](0,0) -- (2,0) -- (1.5,1.5) -- (0.5,1.5) -- cycle;
     \node at (1,-0.5) {$S(p, q_1)$};
 \node[circle,draw,fill=white]  (e) at (2.5,0) {$e$};
 \node (z1) at (0.6, 1.5) {};
 \node (z2) at (0.9, 1.5){};
 \node (z3) at (1.2, 1.5){};
 \draw[-](e) -- (z1);
  \draw[-](e) -- (z2);
  \draw[-](e) -- (z3);
 \fill[gray!30] 
        (3,0) -- (5,0) -- (4.5,1.5) -- (3.5,1.5) -- cycle;
    \draw[-](3,0) -- (5,0) -- (4.5,1.5) -- (3.5,1.5) -- cycle;
     \node at (4,-0.5) {$S(p+1, q_2)$};
   \node (z4) at (4, 1.5){};
   \draw[-](z2) to[out=25,in=155] (z4);
\end{tikzpicture}
& $Fix(S(p, q_1)) +$  \hspace*{2em}$Fix(S(p+1, q_2))$  \\
\hline
$S_4(p,q)$ with $p \ge 1$, $q \ge 1$ & 
\vspace{1em}
\begin{tikzpicture}
\fill[gray!30] 
        (0,0) -- (2,0) -- (1.5,1.5) -- (0.5,1.5) -- cycle;
    \draw[-](0,0) -- (2,0) -- (1.5,1.5) -- (0.5,1.5) -- cycle;
     \node at (1,-0.5) {$S(p,2)$};
 \node[circle,draw,fill=white]  (e) at (2.5,0) {$e$};
  \node[circle,draw,fill=white]  (f) at (2.5,1.5) {$f$};
 \node (z1) at (0.6, 1.4) {};
 \node (z2) at (0.9, 1.4){};
 \node (z3) at (1.2, 1.4){};
 \node (z4) at (0.3,0){};
 \node (z5) at (0.6,0){};
 \node (z6) at (0.9,0){};
 \node (z7) at (1.2,0){};
 \draw[-](e) -- (z1);
  \draw[-](e) -- (z2);
   \draw[-](z1) to[out=25,in=155]   (f); 
  \draw[-](z2) to[out=20,in=160]  (f);  
  \draw[-](z3) to[out=15,in=165]  (f); 
   \draw[-](z4) -- (f);  
  \draw[-](z5) -- (f); 
   \draw[-](z6) -- (f);  
  \draw[-](z7) -- (f); 
  
 \fill[gray!30] 
        (3,0) -- (5,0) -- (4.5,1.5) -- (3.5,1.5) -- cycle;
    \draw[-](3,0) -- (5,0) -- (4.5,1.5) -- (3.5,1.5) -- cycle;
     \node at (4,-0.5) {$S(p+1, q)$};
   \node (z4) at (4, 1.5){};
   \draw[-](z2) to[out=35,in=145] (z4);
   \node (y1) at (3.9, 1.4) {};
 \node (y2) at (4.2, 1.4){};
 \node (y3) at (4.5, 1.4){};
 \node (y4) at (3.5,0){};
 \node (y5) at (3.8,0){};
 \node (y6) at (4.2,0){};
\node (y7) at (4.5,0){};
    \draw[-](f) to[out=15,in=165]  (y1); 
  \draw[-](f) to[out=20,in=160]  (y2);  
  \draw[-](f) to[out=25,in=155]  (y3); 
    \draw[-](y4) -- (f); 
  \draw[-](y5) -- (f);  
  \draw[-](y6) -- (f); 
   \draw[-](y7) -- (f); 
\end{tikzpicture}
& $Fix(S(p,2)) +$ \hspace*{2em}$Fix(S(p+1,q))$  \\
\hline
\end{tabular} 
 \caption{The four types of indecomposable non-split graphs with two or more vertices and their fixing numbers. The shaded vertices in the second column also show the vertices that should be chosen for a fixing set.}  
 \label{SplittableTable2}
 \end{table}

\begin{proposition}
\label{propequalaut}
Let $G$ and $H$ be two graphs with the same set of vertices; i.e. $V(G) = V(H)$.  Furthermore,  $Aut(G) = Aut(H)$.  Then every fixing set of $G$ is a fixing set of $H$ and vice versa so $Fix(G) = Fix(H)$.
\end{proposition}

\begin{proof}
Assume $S$ is a fixing set of $G$ but not of $H$.  Then some $\pi \in Aut(H)$ has the property that $\pi$ assigns each elements in $S$ to itself but for some $v \in \overline{S}$, $\pi(v) \neq v$.   Since $Aut(H) = Aut(G)$, it follows that $\pi \in Aut(G)$ and behaves the same way in $G$, contradicting the fact that $S$ is a fixing set of $G$.   Hence, $S$ is a fixing set of $H$.  The same argument shows that every fixing set of $H$ is also a fixing set of $G$.  Consequently, $Fix(G) = Fix(H)$.  
\end{proof}

\smallskip


\begin{proposition}
For any graph $G$, $Fix(G) = Fix(\overline{G})$.  Additionally, when $G$ is an indecomposable split graph, $Fix(G) = Fix(G^I)$.
\label{equaldistnums}
\end{proposition}

\begin{proof}
For any graph $G$, $Aut(G) = Aut(\overline{G})$.  In Proposition 2.19 of \cite{Ch25}, we also showed that when $G$ is an indecomposable split graph $Aut(G) = Aut(G^I)$.  Thus,  by Proposition \ref{propequalaut},  $Fix(G) = Fix(\overline{G})$ for any graph $G$ and $Fix(G) = Fix(G^I)$ for any indecomposable split graph $G$. 
\end{proof}

Here now is our theorem for the fixing numbers of unigraphs.

\begin{theorem}
When $G$ is a unigraph, $Fix(G)$ can be computed in linear time.  
\end{theorem}

\begin{proof}
Theorem \ref{fixthm} provides a method for computing $Fix(G)$.  First, find the {\it compact} canonical decomposition of $G = (G'_\ell, A_\ell, B_\ell)  \circ \cdots \circ (G'_1, A'_1, B'_1) \circ G'_0$.  Next, for each $G'_i$, determine $Fix(G'_i)$.  Finally, return the sum $Fix(G'_\ell) + \hdots + Fix(G'_1) + Fix(G'_0)$.   

 Both {\tt Decompose($G$)} and {\tt IsUnigraph($G$)} can easily be modified to return the (paired) degree sequence and type of each component $G'_i$ in the compact canonical decomposition of $G$ in linear time.  When $G'_0$ is an indecomposable non-split graph, Table \ref{UnsplittableTable2} provides the answer to $Fix(G_0)$ since $G_0$ is isomorphic to one of the graphs in the table or its complement which,  according to Proposition \ref{equaldistnums}, have the same fixing number.  For the same reason, when  $G'_i$, $i \ge 0$ is an indecomposable split graph with two or more vertices, $Fix(G'_i)$ can be found in Table \ref{SplittableTable2}.   Finally, when $G'_i$ is a complete graph or a graph of isolated vertices, $Fix(G'_i) = |V(G'_i)| - 1$.   Clearly, computing $Fix(G'_i)$ for $i = 0, \hdots, \ell$ can be done in time linear in the size of $G$.  The theorem follows.
\end{proof}

\section{A Toolkit}
\label{sec6}

        To make our discussion more accessible, we have created a toolkit \footnote{\url{https://chelseal11.github.io/tyshkevich_decomposition_toolkit/}} that implements  the algorithms described in this write-up.  The tools can do the following:
        
\begin{itemize}
\item Given a  sequence of graphs, it outputs the composition of the graphs.  The input and output graphs are described by their (paired) degree sequences.

\item Given a sequence of indecomposable unigraphs, it outputs the composition of the unigraphs.  The user selects the type of each indecomposable unigraph from a drop-down menu and specifies the related parameters.

\item Given positive integers $n$ and $k$,  it generates unigraphs with $n$ vertices that have $k$ indecomposable components. 

\item Given the degree sequence of a graph, it outputs the canonical decomposition of the graph.  Each indecomposable component is described by its (paired) degree sequence.  It determines if the input graph is a unigraph; if so, it also specifies the type of each indecomposable component.  It can also determine the  clique number, chromatic number, independence number and vertex cover number of the original graph.

\item Given the degree sequence of a graph, it outputs the {\it compact} canonical decomposition of the graph.  Again, each component is described by its degree sequence.  If the input graph is a unigraph, it also outputs the type of each component.  Additionally, it can determine the distinguishing and fixing numbers of the original graph.

\end{itemize}

\section{Conclusion}
\label{sec7}

	
	We have described Tyshkevich's main results and elaborated on  the algorithmic implications of her work,  showing that  various graph parameters can be computed in linear time when the input graph is a unigraph.   It is very likely that  more polynomial-time algorithms can be designed for unigraphs.  The results here also beg the question whether a similar approach can work for other graph classes.  That is, can their indecomposable components be completely classified and then used for algorithmic purposes?  We hope that our exposition, along with the toolkit we have developed, will encourage readers to engage with this topic and inspire further research.

\bibliography{references}
\bibliographystyle{abbrv}
	
\newpage	
\section*{Appendix}


Below, we present the algorithm {\tt IsUnigraph}($G$).  Given graph $G$, it determines whether $G$ is a unigraph.  If $G$ is a unigraph, it additionally returns a list that contains the type of each indecomposable component of $G$.  
The algorithm relies on {\tt Decompose}($G$) and {\tt DetermineSplit}($\mathit{DegSeq}$), two algorithms that can be found in \cite{Ch25}.  

{\tt Decompose}($G$) returns a stack that contains the (paired) degree sequences of the indecomposable components in the canonical decomposition of $G$.  At the top of the stack is the degree sequence of the first indecomposable component of $G$, which can be a non-split or split graph.  It is followed by the degree sequence of the second indecomposable component and so forth.  {\tt DetermineSplit}($\mathit{DegSeq}$) determines whether the graph with degree sequence $\mathit{DegSeq}$ is a split graph.  If the graph is split, the algorithm returns a paired degree sequence corresponding to a $KS$-partition.  If not, it returns the original $\mathit{DegSeq}$.  


{\tt DetermineNonSplitType}($\mathit{DegSeq}$) assumes that $\mathit{DegSeq}$ is the degree sequence of a  non-split graph.  If the graph is an  indecomposable non-split unigraph, the algorithm returns the graph's type; otherwise, it returns $null$.   {\tt DetermineSplitType}($\mathit{DegSeq}$) behaves the same way  for split graphs.  


\bigskip
\begin{algorithm}[H]
\caption{{\tt IsUnigraph}($G$)}

\begin{algorithmic}[1]
\STATE $T \leftarrow \text{{\tt Decompose}}(G)$
\STATE $componentList \leftarrow \text{[]}$
 \STATE $\mathit{DegSeq} \leftarrow \text{{\tt DetermineSplit}}(T.{top})$
 \IF{$\mathit{DegSeq}$ is not a paired degree sequence}
      \STATE $ans \leftarrow \text{\tt DetermineNonSplitType}(\mathit{DegSeq})$
       \IF{$ans = null$}
            \STATE return (``Not a unigraph")
 	\ELSE
 	     \STATE append $ans$ to $\mathit{componentList}$
	     \STATE $T.{pop}$
	 \ENDIF
\ENDIF
\WHILE{$T.size \ge 1$}
    \STATE $\mathit{DegSeq} \leftarrow T.{pop}$
    \STATE $ans \leftarrow \text{\tt DetermineSplitType}(\mathit{DegSeq})$
    \IF{$ans = null$}
       \STATE return (``Not a unigraph")
     \ELSE   
         \STATE append $ans$ to $\mathit{componentList}$
      \ENDIF
 \ENDWHILE
 \STATE return (``A unigraph", $\mathit{componentList}$)
  
\end{algorithmic}

\end{algorithm}

 
 Let $H$ be some non-split  graph, and $\mathit{DegSeq}$ its degree sequence.  
 {\tt DetermineNonSplitType}($\mathit{DegSeq}$)  considers both $H$ and $\overline{H}$, and checks whether one of them is isomorphic to $C_5, mK_2$, $U_2(m, \ell)$ or $U_3(m)$,  the four indecomposable non-split unigraphs.  If a match is found,  the algorithm returns the information of the match -- the variant of $H$ (original or complement) and the indecomposable non-split unigraph it is matched to.
  Otherwise, the algorithm returns $null$,  indicating $H$ is not an indecomposable non-split unigraph.   We remind the reader that if $\overline{H}$ is isomorphic to say $mK_2$, then $H$ is also isomorphic to $\overline{mK_2}$.  Thus, the information of the match returned by the algorithm is the type of $H$.

  Let $S$ be the degree sequence of $H$ or $\overline{H}$.  To check for a match,  the algorithm compares $S$ with the degree sequences of $C_5, mK_2$, $U_2(m, \ell)$ and $U_3(m)$.  We demonstrate this for $U_2(m, \ell)$.
 From Table \ref{UnsplittableTable}, the degree sequence of $U_2(m, \ell)$ is $(\ell, 1^{2m+\ell})$ with $m \ge 1$, $\ell \ge 2$.  To determine if $S$ is the degree sequence of $U_2(m, \ell)$ for some $m$ and $\ell$, the algorithm first checks if the ``length" of $S$ is $2$; i.e., $S$ has two distinct degrees so $S = (d_1^{r_1}, d_2^{r_2})$ with $d_1 > d_2$. 
 Then it verifies that $r_1 = 1$, $d_2 = 1$ and $r_2 - d_1$ is even.  If so, it sets the parameter $m$ to $(r_2 - d_1)/2$ and $\ell$ to $d_1$.   Finally, it checks whether $m \ge 1$, $\ell \ge 2$.   If all of these conditions are satisfied, then $d_1 = \ell$, $r_1 = 1$, $d_2 = 1$,  $r_2 = 2m + \ell$ with $m \ge 1$, $\ell \ge 2$.  That is, $S = (\ell, 1^{2m+\ell})$,  the degree sequence of $U_2(m, \ell)$.  
 
 The process for $U_3(m)$ is slightly different.  From Table \ref{UnsplittableTable}, the degree sequence of $U_3(m)$ is $(2m+2, 2^{2m+3})$ with $m \ge 1$.  After doing the initial checks on the length of $S$, the parity of $d_1$ and the values of $r_1$ and $d_2$, the algorithm derives the parameter $m$ from $d_1$. Checking if $m \ge 1$  is not enough because $m$ is also related to $r_2$.  It also verifies that $r_2 = 2m+3$.  In this way, $(d_1^{r_1}, d_2^{r_2}) = (2m+2, 2^{2m+3})$ for some $m \ge 1$.




\begin{algorithm}[H]
\caption{{\tt DetermineNonSplitType}($\mathit{DegSeq}$)}
\begin{algorithmic}[1]
	\STATE $\mathit{variants} \leftarrow [\text{``original",  ``complement"}]$
     	\STATE $\mathit{tests} \leftarrow  [\text{``C5", ``MK2", ``U2", ``U3"}]$
    	\FOR{$variant$ \text{in} $\mathit{variants}$}
      		\STATE $S  \leftarrow {applyVariant}(variant, \mathit{DegSeq})$\;
      		\FOR{$test$ \text{ in } tests}
		   \IF{$test = \text{``C5"}$ and $S = (2^5)$}
		   	\STATE return ($variant$,  \text{``C5"})
		    \ENDIF
		     \IF{$test = \text{``MK2"}$ and $length(\mathit{S}) =1$ }
		         \STATE $(d_1^{r_1}) \leftarrow \mathit{S}$
		         \IF{$d_1 = 1$ and $r_1$ is even}
		           \STATE $\mathit{m} \leftarrow r_1 / 2$
		           \IF {$m \ge 2$}
		               \STATE return ($variant$,  \text{``MK2"}, $m$)
		            \ENDIF
		          \ENDIF
		       \ENDIF       
		     \IF{$test = \text{``U2"}$ and $length(\mathit{S}) =2$ }  
		        \STATE $(d_1^{r_1}, d_2^{r_2}) \leftarrow \mathit{S}$ with $d_1 > d_2$
		        \IF{$r_1 = 1$  and $d_2 = 1$ and $(r_2 - d_1)$ is even}
		           \STATE $m \leftarrow (r_2 - d_1) / 2$, $\ell \leftarrow d_1$
		           \IF{$m \ge 1$ and $\ell \ge 2$}
		              \STATE return ($variant$,  \text{``U2"}, $m, \ell$)
		            \ENDIF
		         \ENDIF
		       \ENDIF
		       \IF{$test = \text{``U3"}$ and $length(\mathit{S}) =2$ }  
		         \STATE $(d_1^{r_1}, d_2^{r_2}) \leftarrow \mathit{S}$ with $d_1 > d_2$
		         \IF{$d_1$ is even and $r_1 = 1$  and $d_2 = 2$}
		            \STATE $\mathit{m} \leftarrow (d_1 - 2) / 2$
		            \IF{$m \ge 1$ and $r_2 = 2m+3$}
		                \STATE return ($variant$,  \text{``U3"}, $m$)
		             \ENDIF
		          \ENDIF
		        \ENDIF
		\ENDFOR
	\ENDFOR
        \STATE return($null$)
\end{algorithmic}
\end{algorithm}

\begin{algorithm}[H]
\caption{{\tt DetermineSplitType}($\mathit{DegSeq}$)}
\begin{algorithmic}[1]
	\STATE $\mathit{variants} \leftarrow [\text{``original", ``inverse", ``complement", ``inverseComp"}]$
     	\STATE $\mathit{tests} \leftarrow [\text{``SVG", ``SPQ", ``S2", ``S3", ``S4"}]$
    	\FOR{$variant$ \text{in} $\mathit{variants}$}
      		\STATE $S = (D_A; D_B) \leftarrow {applyVariant}(variant, \mathit{DegSeq})$\;
      		\FOR{$test$ \text{ in } tests}
			\IF{$test = \text{``SVG"}$}
	 	                \IF{$length(D_A) = 1$ and $length(D_B) = 0$}
                  			\STATE $(d_1^{r_1}) \leftarrow D_A$\;
                  			\IF{$d_1 = 0$ and $r_1 = 1$}
                    				\STATE return($variant$,  \text{``SVG-K1"})
                  			\ENDIF
			           \ENDIF
		                  \IF{$length(D_A) = 0$ and $length(D_B) = 1$}
                  			\STATE $(d_2^{r_2}) \leftarrow D_B$\;
                 			 \IF{$d_2 = 0$ and $r_2 = 1$}
                    			     \STATE return($variant$,  \text{``SVG-S1"})
	             			\ENDIF
			           \ENDIF
			     \ENDIF   	
			  \IF{$test = \text{``SPQ"}$ and $length(D_A) = 1$ and $length(D_B) = 1$}
	                	     	\STATE $(d_1^{r_1}) \leftarrow D_A$, $(d_2^{r_2}) \leftarrow D_B$
                		      		\IF{$r_2 \bmod r_1 = 0$ and  $d_2  = 1$}      						          
		               			\STATE $p \leftarrow r_2 / r_1$, $q \leftarrow r_1$
		               				\IF{$p \ge  1$ and $q \ge 2$ and $d_1 = p+q-1$}
							     \STATE return($variant$, \text{``SPQ"}, $p, q$)
							 \ENDIF
				          \ENDIF
		             \ENDIF    
		             
		           \IF{$test = \text{``S2"}$ and $length(D_A) \ge 2$ and $length(D_B) = 1$}
		                   \STATE  $(d_1^{r_1}, d_2^{r_2}, \dots, d_m^{r_m}) \leftarrow D_A$ with $d_1 > d_2 > \dots > d_m$ and $(d_B^{r_B}) \leftarrow D_B$
		                    \STATE $n \leftarrow r_1 + r_2 + ... + r_m$
		                       \FOR{$i \leftarrow 1$ to $m$}
                  			 \STATE $p_i \leftarrow d_i - n + 1$, $q_i \leftarrow r_i$
		                       \ENDFOR
		                   \IF{$p_m \ge 1$}
		                       \IF {$d_B = 1$ and $r_B = \sum_{i=1}^m p_i q_i$}
		                          \STATE return($variant$, \text{``S2"}, ($p_1, q_1, p_2, q_2, \dots, p_m, q_m$))
		                        \ENDIF
		            	    \ENDIF
			      \ENDIF	
			
                           \IF{$test = \text{``S3"}$ and $length(D_A) = 1$ and $length(D_B) = 2$}                			
                			\STATE $(d_1^{r_1}) \leftarrow D_A$, $(d_2^{r_2}, d_3^{r_3}) \leftarrow D_B$ with $d_2 > d_3$
			        \IF{$r_2 = 1$ and $d_3 = 1$}
			         	\STATE $p \leftarrow d_1 - r_1$, $q_1 \leftarrow d_2$, $q_2 \leftarrow r_1 - d_2$
			                 \IF{ $p \ge 1$ and $q_1 \ge 2 $ and $q_2 \ge  1$}
			        		           \IF{$r_3 = pq_1 + (p+1)q_2$}
				                 \STATE return($variant$, \text{``S3"}, $p, q_1, q_2$)
				                 \ENDIF
				            \ENDIF
				        \ENDIF
		             \ENDIF	
		             
		           \IF{$test = \text{``S4"}$ and  $length(D_A) = 2$ and $length(D_B) = 1$}
		               \STATE $(d_1^{r_1}, d_2^{r_2}) \leftarrow D_A$ with $d_1 > d_2$, $(d_3^{r_3}) \leftarrow D_B$
		               \IF{$r_1 = 1$ and $d_3 = 2$}
		                   \STATE $q \leftarrow r_2 - 2$, $p \leftarrow d_2 - 3 - q$
		                   \IF{$p \ge 1$ and $q \ge 1$}
		                   	\IF{ $d_1 = 2(p + q + 1) + qp$ and $r_3 = qp + 2p + q + 1$}
		           			\STATE return($variant$, \text{``S4"}, $p, q$)
					\ENDIF
				  \ENDIF
			       \ENDIF
		            \ENDIF         
		\ENDFOR
	\ENDFOR
         \STATE return($null$)
\end{algorithmic}
\end{algorithm}

This time, let $H$ be some split  graph and $\mathit{DegSeq}$ be its degree sequence.  The algorithm {\tt DetermineSplitType}($\mathit{DegSeq}$) works the same way as {\tt DetermineNonSplitType} by comparing the degree sequence of $H$, $\overline{H}$, $H^I$, $\overline{H}^I$ with those of a single vertex graph, $S(p,q)$, $S_2(p_1, q_1, \hdots, p_k, q_m)$, $S_3(p, q_1, q_2)$ and $S_4(p,q)$.   The algorithm first checks if the degree sequence has the correct basic format.  It then derives the parameters and checks if the parameters are valid.  Finally, if certain parts of the degree sequence are based on the parameters but were not used in the derivation of the parameters,  the algorithm checks if they have the right values.  
 
 Let $S$ be the degree sequence of $H$, $\overline{H}$, $H^I$ or $\overline{H}^I$.  Let us demonstrate the above steps when the algorithm checks if $S$ is isomorphic to   $S_2(p_1, q_1, \hdots, p_m, q_m)$ for some $p_i, q_i$, $i = 1, \hdots, m$.  From Table \ref{SplittableTable}, the degree sequence of the latter is $((p_1+ n -1)^{q_1}, \cdots, $ $(p_m + n-1)^{q_m};$ $1^{p_1q_1 + \cdots + p_m q_m})$ where $n = \sum_{i=1}^m q_i$ with $m \ge 2$, $p_1 > p_2 > \hdots > p_m \ge 1$ and $q_i \ge 1$ for $i = 1, \hdots, m$.  
 
 First, the algorithm checks if in the paired degree sequence $(D_A, D_B)$,  $D_A$ has length at least $2$ and $D_B$ has length $1$.   If so, its format is  $(d_1^{r_1}, d_2^{r_2}, \dots, d_m^{r_m}; d_B^{r_B})$ with $m \ge 2$ and $d_1 > d_2 > \hdots > d_m$.  It then sets the parameters $n = \sum_{i=1}^m r_i$,  $p_i = d_i - n+1$ and $q_i = r_i$ for $i = 1, \hdots, m$.   Next, it verifies that $p_m \ge 1$ and $d_B = 1$.  There is no need to check if $q_i \ge 1$ since $r_i \ge 1$. 
 Finally, it checks that $r_B = \sum_{i=1}^m p_i q_i$.    If all these conditions are satisfied, then $d_i = p_i + n -1$ and $r_i = q_i$ for $i = 1, \hdots, m$ with $n = \sum_{i=1}^m q_i$.   Since the $d_i$'s are in decreasing values, so are the  $p_i$'s with $p_m \ge 1$.   It is also the case that $d_B^{r_B} = (1^{p_1q_1 + \cdots + p_m q_m})$.

 In both {\tt DetermineNonSplitType} and {\tt DetermineSplitType}, the number of variants of $H$ and the number of graphs a variant of $H$ is compared to are constants.  Deriving the degree sequence of each variant of $H$ takes $O(|V(H)|)$ time. The time it takes to check if a variant of $H$ is isomorphic to some indecomposable unigraph is $O(|V(H)|)$.  Thus, both algorithms run in $O(|V(H)|)$ time.  This means that in {\tt IsUnigraph},  lines 1 to 18 take $O(|V(G)|)$ time.  Both {\tt Decompose} and {\tt DetermineSplit} also run in linear time.  Hence, {\tt IsUnigraph} runs in linear time.
 
 \end{document}